\documentclass[pdflatex,sn-mathphys-num]{sn-jnl}

\usepackage[american]{babel}
\usepackage{amsmath,amssymb,amsthm,amsfonts,mathrsfs}
\usepackage{bm,bbm}
\usepackage{graphicx}
\usepackage{booktabs}
\usepackage{float}
\usepackage{microtype}
\usepackage{xcolor}
\definecolor{mycolor}{RGB}{25,25,112}
\hypersetup{colorlinks=true,linkcolor=mycolor,citecolor=mycolor,urlcolor=mycolor}

\newtheorem{Def}{Definition}[section]
\newtheorem{lem}[Def]{Lemma}
\newtheorem{theo}[Def]{Theorem}
\newtheorem{pro}[Def]{Proposition}
\newtheorem{rem}[Def]{Remark}

\newtheorem{assum}{Assumption}

\newcommand{\LL}{\langle}
\newcommand{\RR}{\rangle}

\newcommand{\mcal}{\mathcal}
\newcommand{\DW}{\Delta W}

\newcommand{\mbb}{\mathbb}
\newcommand{\mbf}{\mathbf}

\newcommand{\ud}{\mathrm d}

\newcommand{\PD}{\partial}
\numberwithin{equation}{section}
\allowdisplaybreaks[4]
\begin{document}

\title[Asymptotic errors via appurtenant methods]{Asymptotic error distributions for diffusion-implicit schemes via appurtenant methods}

\author*[1,2]{\fnm{Diancong} \sur{Jin}}\email{jindc@hust.edu.cn}

\affil[1]{\orgdiv{School of Mathematics and Statistics}, \orgname{Huazhong University of Science and Technology}, \orgaddress{\city{Wuhan}, \postcode{430074}, \country{China}}}
\affil[2]{\orgdiv{Hubei Key Laboratory of Engineering Modeling and Scientific Computing}, \orgname{Huazhong University of Science and Technology}, \orgaddress{\city{Wuhan}, \postcode{430074}, \country{China}}}

\abstract{We develop an appurtenant-method framework for analyzing asymptotic error distributions and implement it in two diffusion-implicit settings for stochastic differential equations. The argument rests on an explicit appurtenant method which has the same strong order as the implicit method and approximates it with a higher strong order. A local-to-global strong convergence result for two Markov chains turns the one-step expansion of the implicit equation into the required appurtenant method; the limiting distribution can then be obtained from an adapted continuization of the appurtenant method. We first illustrate the framework with the fully implicit stochastic midpoint method for a multidimensional Stratonovich equation, without imposing commutativity on the diffusion fields. This method has strong order $1/2$, and its normalized error is governed by the L\'evy areas omitted by the numerical method. We then turn to the main application, namely, a class of fully implicit stochastic Runge--Kutta (SRK) methods of strong order $1$ for Stratonovich equations with scalar multiplicative noise. We derive their asymptotic error distributions and identify a parameter, depending only on the coefficients of the Butcher tableau, that reflects the growth rate of the root-mean-square error of the SRK method.}


\maketitle
\section{Introduction}
In this work, we develop an appurtenant-method framework for analyzing asymptotic error distributions and implement it in two diffusion-implicit settings. We use the following Stratonovich-type stochastic differential equation (SDE) as the model problem:
\begin{align}\label{SDE0}
	\begin{cases}
		\ud Y(t)=f(Y(t))\ud t+g(Y(t))\circ\ud \mbf W(t),~t\in(0,T], \\
		Y(0)=Y_0\in\mbb R^d,
	\end{cases}
\end{align}
where $\{\mbf W(t)=(W_1(t),\ldots,W_m(t))^\top,~t\in[0,T]\}$,  is an $m$-dimensional standard Brownian motion defined on $(\Omega,\mcal F,\mbf P)$ with respect to (w.r.t.) a filtration $\{\mcal F_t\}_{t\in[0,T]}$ satisfying the usual conditions, and $f:\mbb R^d\to\mbb R^d$, $g:\mbb R^d\to\mbb R^{d\times m}$ are Lipschitz continuous such that \eqref{SDE0} admits a unique strong solution. 
Stochastic numerical methods have played an important role in numerically approximating SDEs and modeling their intrinsic behavior over the past decades \cite{Milsteinbook}.
In the theoretical study on stochastic numerical methods, the error analysis, measuring the reliability and accuracy of stochastic numerical methods, is one of the most basic and most important subject. Compared with the fruitful results on the convergence analysis of stochastic numerical methods, including the strong convergence, weak convergence and  convergence in probability,  the investigation on probabilistic limit theorems for errors between numerical solutions and exact solutions has been far from sufficient. In this paper, we are interested in the asymptotic error distribution of stochastic numerical methods applied to the SDE \eqref{SDE0}, which is a kind of generalized central limit theorem.

\subsection{Research background}
The asymptotic error distribution of stochastic numerical methods is the limit distribution  of the normalized error between numerical solutions and exact solutions as the step-size tends to zero, where the normalized error is the error weighted by the corresponding strong convergence order of numerical methods. The asymptotic error distribution provides the optimal strong convergence order for stochastic numerical methods and characterizes the evolution pattern of distribution of  the error in the small step-size regime. It also finds some applications in  the error structure of stochastic numerical methods \cite{Bouleau} and deriving the central limit theorem of stochastic algorithms \cite{CLTMonte15}. The study on the asymptotic error distribution dated back to the work \cite{Protter1991} for SDEs with bounded coefficients, where the authors gave a sufficient condition for  Euler method to  admit the asymptotic error distribution. Subsequent research by \cite{Protter1998AOP} provided a sufficient and necessary condition for the asymptotic error distribution of the Euler method when applied to
SDEs with linearly growing coefficients. Further, \cite{Protter2020SPA} established the   asymptotic error distribution of the Euler method for SDEs with locally Lipschitz
coefficients. 
We also refer to the asymptotic error distribution of numerical methods applied to stochastic integral equations \cite{Fukasawa2023,Nualart2023}, Mckean--Vlasov SDEs \cite{Wufuke}, and SDEs driven by fractional Brownian motions \cite{HuAAP,HuCrankNicolson,HuBIT}. 

Most of the aforementioned works consider the asymptotic error distribution of
explicit methods for SDEs. To the best of our knowledge, in the existing literature,
there are no results on the asymptotic error distribution of diffusion-implicit numerical methods (the implicitness is introduced via the diffusion fields) for \eqref{SDE0}, despite
their superior stability and frequent applications in constructing structure-preserving algorithms. The aim of this paper is to address this gap through a common appurtenant-method argument, implemented in two settings with different noise structures and strong convergence orders. Generally speaking, when establishing the asymptotic error distribution of a numerical method for \eqref{SDE0}, one usually finds a continuous version of the numerical method and decomposes the normalized error into dominant and remainder terms. Jacod's theory on convergence in distribution of conditional Gaussian martingales \cite{Fukasawa2023,Jacod97,Protter1998AOP} then plays a key role in studying the dominant stochastic integrals. For a diffusion-implicit or fully implicit method, however, a direct continuization generally produces non-adapted integrands. In the two applications considered here, the framework bypasses this issue by transferring the normalized-error limit from the implicit method to a fully explicit appurtenant method.

\subsection{An appurtenant-method framework}
We summarize in Framework \ref{frame} the common passage from an implicit method to its appurtenant method used in the two applications below. We then apply the construction to two different situations. The first is the fully implicit stochastic midpoint method of strong order $1/2$ for $d$-dimensional Stratonovich equations driven by $m$ general diffusion fields. The second, and principal, application is a class of fully implicit SRK methods of strong order $1$ for scalar multiplicative noise. To circumvent the difficulty of directly continuizing an implicit method, we look for a fully explicit method which shares its asymptotic error distribution. Such an explicit method needs to approximate the exact solution with the strong order of the original method and to approximate the implicit numerical solution with a strictly higher strong order. We call it an appurtenant method. To construct one, we generalize the fundamental strong convergence theorem \cite[Theorem 2.1]{ZhangZQ13} to estimate the strong error between any two Markov chains (see Theorem \ref{fundamental}) and combine that estimate with a deterministic Taylor expansion of the implicit one-step equation. Beyond these two settings, applying the same strategy requires a scheme-specific one-step expansion and a higher-order comparison estimate.
The proof pattern used in the two settings is summarized as follows.
\begin{table}[htbp]
	\caption{Framework $1$ for establishing the asymptotic error distribution of $Y_N$.}
	\label{frame}
	\centering
	\small
	\renewcommand{\arraystretch}{1.22}
	\begin{tabular}{@{}p{0.105\textwidth}p{0.835\textwidth}@{}}
		\toprule
		\textbf{Step 1} & Determine the strong convergence order $p$ of
		$\{Y_n\}_{n=0}^N$. \\
		\addlinespace[2pt]
		\textbf{Step 2} & Find the explicit appurtenant method
		$\{\widetilde Y_n\}_{n=0}^N$ based on Theorem \ref{fundamental} and a
		Taylor expansion. The method must approximate the solution to \eqref{SDE0}
		with strong order $p$ and approximate $\{Y_n\}_{n=0}^N$ with strong order
		$q>p$. \\
		\addlinespace[2pt]
		\textbf{Step 3} & Construct a suitable continuization of
		$\{\widetilde Y_n\}_{n=0}^N$ and establish the limit distribution of
		$N^p(\widetilde Y_N-Y(T))$, which is also the limit distribution of
		$N^p(Y_N-Y(T))$. \\
		\bottomrule
	\end{tabular}
\end{table}
\noindent To explain the last step in Framework \ref{frame}, write
$	N^p\big(Y_N-Y(T)\big)
	=N^p\big(Y_N-\widetilde Y_N\big)
	+N^p\big(\widetilde Y_N-Y(T)\big).$
If $\widetilde Y_N$ approximates $Y_N$ with strong order $q>p$, the first term on the right-hand side converges to zero in probability. Hence any limit in distribution, or stable limit, established for the normalized error of the appurtenant method is inherited by the original implicit method, by Slutsky's theorem \cite[Theorem 3.1]{Billing1992}.

\subsection{Main results and contributions}
The technical basis of the framework is Theorem \ref{fundamental}, which turns suitable local defects into a global strong error estimate between two Markov chains. Combined with a Taylor expansion of the implicit one-step equation, this result provides the higher-order comparison required in Step 2 of Framework \ref{frame}. Theorem \ref{stablecon} is then used to pass from the stable limit of the dominant perturbation to the corresponding linearized error equation.

Theorem \ref{midpointlimit} gives the first application of the framework. It concerns the fully implicit stochastic midpoint method under general multidimensional noise. In this case the method has strong order $1/2$, the normalization is $\sqrt N$, and the limiting perturbation is produced by the missing L\'evy areas. We subsequently develop the core application: the asymptotic error distribution of the fully implicit SRK method \eqref{SRK1} of strong order $1$ for scalar multiplicative noise, i.e., $N(Y_N^{SRK}-Y(T))\overset{d}{\Rightarrow}U(T)$ with $\{U(t)\}_{t\in[0,T]}$ satisfying a linear SDE; see Theorem \ref{maintheorem1}. 
As an application of Theorem \ref{maintheorem1},  we show in Theorem \ref{upperbound} that $\frac{(\mbf E|U(T)|^2)^{1/2}}{T}$ is the asymptotic root-mean-square constant of the SRK method \eqref{SRK1}. Further, we establish in Theorem \ref{upperbound} the upper bound $\mbf E|U(T)|^2\le e^{L_1T}(1+\eta_1)T^3$. Here, $L_1$ depends only on $f,g,Y_0$, whereas $\eta_1\ge 0$ depends only on the coefficients in the Butcher tableau of the SRK method \eqref{SRK1}. Thus, we infer that  $\eta_1$ is the key parameter reflecting the growth rate of the root-mean-square error of the SRK method \eqref{SRK1}. In addition, we show that the SRK method \eqref{SRK1} is of weak order $2$ provided that $\eta_1=0$. Hence, among the strong-order-one SRK methods \eqref{SRK1}, methods of weak order $2$ minimize this upper bound. In particular, Theorem \ref{maintheorem2} shows that 
those SRK methods of weak order $2$ share the same asymptotic error distribution. 

Let us summarize the main contributions of this work as follows.
\begin{itemize}
	\item  We formulate an  appurtenant-method framework for diffusion-implicit or fully implicit schemes.  Since the construction relies on an appropriate one-step expansion rather than on the specific form of an integrator, the same strategy can, in principle, be extended to higher-order numerical methods.
	
	\item We apply the framework to the fully implicit stochastic midpoint method under general multidimensional noise. The method is of strong order $1/2$, and the limiting error is determined by the Lie brackets of the diffusion fields and the L\'evy areas omitted by the midpoint rule.
	
	\item We identify the key parameter reflecting the growth rate of the root-mean-square error of the SRK method by analyzing the properties of the limit distribution. Our results reveal that among the SRK methods  of strong order $1$, those of weak order $2$ have the smallest  mean-square errors after a long time. Especially, the asymptotic error distribution of the considered SRK methods takes the unified form, when endowed with second order weak convergence.   
\end{itemize}

The rest of the article is organized as follows. Section \ref{Sec2} establishes the Markov-chain strong error estimate and a stable-convergence result for perturbed SDEs. Section \ref{SecMidpoint} illustrates the framework with the fully implicit stochastic midpoint method under general multidimensional noise. Sections \ref{SecSRK} and \ref{SecAsymptoticSRK} then develop the main application to fully implicit SRK methods for scalar multiplicative noise. Section \ref{SecNumerical} presents a numerical test of the limit distribution for the multidimensional midpoint method and a long-time error comparison for a parametric SRK family. Section \ref{SecConclusion} concludes the paper.

\section{Preliminaries}\label{Sec2}
In this section, we present a local-to-global strong convergence theorem for two Markov chains and a stable-convergence theorem for a class of perturbed SDEs. The former will be used to construct the appurtenant methods in the two applications below, whereas the latter will be used in identifying their normalized-error limits.

We begin with some notations.  Denote by $|\cdot|$  the $2$-norm of a vector or matrix, and by $\LL\cdot,\cdot\RR$ the scalar product of two vectors.  Denote by $\mbf L^p(\Omega,\mcal F,\mbf P;\mbb R^d)$, $p\ge 1$, the Banach space consisting of $p$th integrable $\mbb R^d$-valued random variables $X$, endowed with the usual norm $\|X\|_{\mbf L^p(\Omega)}:=(\mbf E|X|^p)^{1/p}$. Moreover, $\overset{d}{\Longrightarrow}$ denotes the convergence in distribution of random variables. Denote by $\LL X, Y\RR_t$, $t\in[0,T]$ the cross variation process between the real-valued semi-martingales $\{X(t)\}_{t\in[0,T]}$ and $\{Y(t)\}_{t\in[0,T]}$.

Denote by $\mathbf C(\mathbb R^d)$ (resp.\ $\mathbf C^{k}(\mathbb R^d)$) the space of continuous (resp.\ $k$th continuously differentiable) functions defined on $\mathbb R^d$. Let $\mathbf C_b^k(\mathbb R^d)$, $k\in\mathbb N^+$,  stand for the subset of $\mathbf C^{k}(\mathbb R^d)$, consisting of functions whose derivatives of positive orders up to \(k\) are bounded. 
For a real-valued function $f\in\mathbf C^k(\mathbb R^d)$, denote by  $\mathcal D^k f(x)(\xi_1,\ldots,\xi_k)$ the $k$th order G\v ateaux derivative along the directions $\xi_1,\ldots,\xi_k\,\in\mathbb R^d$. 
For an $\mathbb R^m$-valued function $f=(f_1,\ldots,f_m)^\top\in\mathbf C^k(\mathbb R^d)$, denote $\mathcal D ^kf(x)(\xi_1,\ldots,\xi_k):=(\mathcal D ^kf_1(x)(\xi_1,\ldots,\xi_k),\ldots,\mathcal D ^kf_m(x)(\xi_1,\ldots,\xi_k))^\top$.
Let $\mathbf F$ be the set of functions growing at most polynomially, i.e., a (tensor-valued) function  $f\in\mathbf F$ means that there exist $C>0$ and $\eta>0$ such that for any $x\in\mathbb R^d$, $|f(x)|\le C(1+|x|^\eta)$ or $\|f(x)\|_{\otimes}\le C(1+|x|^\eta)$, where $\otimes$ denotes the norm of a tensor. Throughout this paper, let $K(a_1,a_2,\ldots,a_m)$ be some generic constant dependent on the parameters $a_1,a_2,\ldots,a_m$ but independent of the step-size $h$, which may vary for each appearance.

\subsection{Fundamental strong convergence theorem}
In this part, we present  a  theorem for analyzing the strong error order between two Markov chains, which will be used to construct the appurtenant method of \eqref{SRK1}.
Let $\{t_0<t_1<\cdots<t_N=T\}$, $N\in\mbb N^+$ be a uniform partition of $[t_0,T]$, i.e., $t_k-t_{k-1}=\frac{T-t_0}{N}:=h$ for any $k=1,\ldots,N$.  To make the dependence on the whole Brownian increment path explicit, set
$\mathcal C_m:=\mathbf C([0,1];\mbb R^m)$ and 
$\mbf W_{t,h}(r):=\mbf W(t+rh)-\mbf W(t)$,  $r\in[0,1]$.
For each $h>0$, introduce two one-step mappings $\bar X_{t,x}(t+h)$ and $\widetilde{X}_{t,x}(t+h)$ depending on $t,x,h$, which are defined respectively by
\begin{align*}
\bar X_{t,x}(t+h)=x+\bar A\bigl(t,x,h;\mbf W_{t,h}\bigr), \quad
\widetilde X_{t,x}(t+h)=x+\widetilde A\bigl(t,x,h;\mbf W_{t,h}\bigr),
\end{align*}
where $\bar A$ and $\widetilde A$ are jointly Borel measurable maps
$\bar A,\widetilde A:\mbb R\times\mbb R^d\times(0,\infty)\times\mathcal C_m\rightarrow\mbb R^d.$
Thus, the one-step maps may depend on the entire Brownian path segment, including iterated stochastic integrals.  Then one can construct two Markov chains $(\bar X_k,\mcal F_{t_k})$ and  $(\widetilde X_k,\mcal F_{t_k})$, $k=0,1,\ldots,N$: Let $\bar X_0=\widetilde{X}_0\in\mbf L^2(\Omega,\mcal F,\mbf P;\mbb R^d)$, and for $k=0,\ldots,N-1$:
\begin{align*}
	\bar X_{k+1}&=\bar{X}_{t_k,\bar{X}_k}(t_{k+1})
=\bar{X}_k+\bar A\bigl(t_k,\bar X_k,h;\mbf W_{t_k,h}\bigr),\\
	\widetilde X_{k+1}&=\widetilde{X}_{t_k,\widetilde{X}_k}(t_{k+1})
=\widetilde{X}_k+\widetilde A\bigl(t_k,\widetilde X_k,h;\mbf W_{t_k,h}\bigr).
\end{align*}

The fundamental strong convergence result \cite[Theorem 2.1]{ZhangZQ13} applies to numerical methods for SDEs. Following its proof, we derive the strong error order between $\bar{X}_k$ and $\widetilde{X}_k$.     

\begin{theo}\label{fundamental}
	Let $\{\bar{X}_k\}_{k=0}^N$ and $\{\widetilde{X}_k\}_{k=0}^N$ be defined as above. Assume that the following conditions hold:
	\begin{itemize}
		\item [(A1)] For some $p\ge 1$, there exist $p_2\ge \frac{1}{2}$, $p_1\ge p_2+\frac{1}{2}$,  $\alpha\ge 1$, $h_0\in(0,1]$ and $K_1>0$ such that for all $t\in[t_0,T-h]$, $x\in\mbb R^d$ and $h\in(0,h_0]$,
		\begin{align*}
			\big|\mbf E\big(\bar{X}_{t,x}(t+h)-\widetilde{X}_{t,x}(t+h)\big)\big|&\le K_1(1+|x|^{2\alpha})^{1/2}h^{p_1}, \\
			\big[\mbf E\big|\bar{X}_{t,x}(t+h)-\widetilde{X}_{t,x}(t+h)\big|^{2p}\big]^{1/(2p)}&\le K_1(1+|x|^{2\alpha p})^{1/(2p)}h^{p_2}.
		\end{align*}  
		\item [(A2)] For any $q\ge 1$, there is $K_2(q)>0$ such that for all $h\in(0,h_0]$ and  $k=0,1,\ldots,N$,
		\begin{align*}
			\mbf E|\bar X_k|^q+\mbf E|\widetilde X_k|^q \le K_2(q).
		\end{align*}	
		\item [(A3)] There is a representation
	$\widetilde{X}_{t,x}(t+h)-\widetilde{X}_{t,y}(t+h)=x-y+V$
		for which
		\begin{gather*}
			\mbf E\big|\widetilde{X}_{t,x}(t+h)-\widetilde{X}_{t,y}(t+h) \big|^{2p}\le |x-y|^{2p}(1+Kh),\\
			\mbf E |V|^{2p}\le K(1+|x|^{p\beta}+|y|^{p\beta})|x-y|^{2p}h^p,
		\end{gather*}
		where $\beta\ge1$ is a constant.
	\end{itemize}
	Then, for any $h\in(0,h_0]$ and $k=0,\ldots,N$,
		$\big[\mbf E\big|\bar X_k-\widetilde X_k\big|^{2p}\big]^{1/(2p)}\le Kh^{p_2-1/2}.$
\end{theo}
\begin{proof}
	This proof follows the same steps as in the proof of \cite[Theorem 2.1]{ZhangZQ13}, by replacing $X$ therein by $\widetilde{X}$ here. We note that \cite[Assumption 2.1]{ZhangZQ13} is used to establish \cite[Lemma 2.2]{ZhangZQ13} and the moment boundedness of $X(t_k)$ therein.  In fact, in the proof of \cite[Theorem 2.1]{ZhangZQ13}, only the one-step approximation error \cite[Eq.\ (2.9)-(2.10)]{ZhangZQ13}, the moment uniform boundedness of $X(t_k)$ and $\bar X_k$ therein, and \cite[Lemma 2.2]{ZhangZQ13} are used. These three conditions correspond to (A1), (A2) and (A3), respectively. Thus, the proof of this theorem can be completed by following the proof procedure of \cite[Theorem 2.1]{ZhangZQ13}.
\end{proof}


\subsection{Stable convergence in distribution}
In this part, we introduce the notion of stable convergence in distribution and giving a theorem on convergence in distribution of  solutions to a sequence of SDEs perturbed by processes stably converging in law.   

Let $\{X_n\}_{n=1}^\infty$ be a sequence of random variables with values in a Polish space $E$, all defined on $(\Omega,\mathcal{F},\mathbf{P})$. Let $(\tilde{\Omega},\tilde{\mathcal{F}},\tilde{\mathbf{P}})$  be an extension of $(\Omega,\mathcal{F},\mathbf{P})$ and let $X$ be an $E$-valued random variable on the extension.
We say that $X_n$ stably converges in law to $X$ in $E$, written $X_n\overset{stably}{\Longrightarrow}X$ in $E$, if 
$\mathbf{E}[Z f(X_n)]\to \tilde{\mathbf{E}}[Z f(X)]$
for all bounded and continuous $f:E\to \mathbb{R}$ and all bounded random variable $Z$ on $(\Omega,\mathcal{F})$, where $\tilde{\mathbf E}$ denotes the expectation w.r.t. $\tilde{\mathbf{P}}$. From the above definition, we know that $X_n\overset{stably}{\Longrightarrow}X$ implies $X_n\overset{d}{\Longrightarrow}X$.
We refer readers to \cite{Jacod97} for some basic properties of stable convergence in law.  A useful property about stable convergence is as follows.
\begin{pro}\cite[Lemma 2.1]{Protter1998AOP}\label{Pro1}
	Let $Y$ be a random variable on another Polish space $F$. If $X_n\overset{stably}{\Longrightarrow}X$ in $E$, then $(Y,X_n)\overset{stably}{\Longrightarrow}(Y,X)$ in $F\times E$.
\end{pro}
On basis of Proposition \ref{Pro1}, we present the following convergence in distribution of a family of solutions to systems of SDEs.  

\begin{theo}\label{stablecon}	
	Let $\Phi^N=\{\Phi^N(t),t\in[0,T]\}$, $N\ge 1$ be	 $\{\mcal F_t\}$-adapted $\mbb R^{d_2}$-valued  processes defined on $(\Omega,\mcal F,\mbf P)$, and $Y=\{Y(t),t\in[0,T]\}$ be an $\{\mcal F_t\}$-adapted $\mbb R^{d_1}$-valued  process defined on $(\Omega,\mcal F,\mbf P)$ with almost sure continuous trajectories.
	Let $Z^N\in \mbf C([0,T];\mbb R^{d_2})$, $N\ge1$, be the strong solution of 
	\begin{equation}\label{sec2eq0}
		Z^N(t)
		=\int_0^tB_0(Y(s))Z^N(s)\ud s+\sum_{i=1}^{m}\int_0^tB_i(Y(s))Z^N(s)\ud W_i(s)+\Phi^N(t),~t\in[0,T],
	\end{equation}
	where  $B_i:\mbb R^{d_1}\to\mbb R^{d_2\times d_2}$ are bounded and $\mcal DB_i\in\mbf F$ (i.e., $\mcal DB_i$ grow at most polynomially) for $i=0,\ldots,m$. 
	Assume that the following conditions hold.
	\begin{itemize}
		\item [(1)] There is $\beta>0$ such that $Y\in\mbf C([0,T];\mbf L^p(\Omega;\mbb R^{d_1}))$ and $\|Y(t)-Y(s)\|_{\mbf L^p(\Omega)}\le K|t-s|^{\beta}$  for  sufficiently large $p$. 
		
		\item [(2)]  $\Phi^N\overset{stably}{\Longrightarrow}\Phi$ in $\mbf C([0,T];\mbb R^{d_2})$, where $\Phi$ is an $\{\tilde{\mcal F_t}\}$-adapted process defined on $(\tilde{\Omega},\tilde{\mathcal{F}},\{\tilde{\mcal F_t}\}_{t\in[0,T]},\tilde{\mathbf{P}})$, an extension of  $(\Omega,\mathcal{F},\{\mcal F_t\}_{t\in[0,T]}, \mathbf{P})$.
		
		\item [(3)]    $\sup\limits_{N\ge1}\sup\limits_{t\in[0,T]}\mbf E|\Phi^N(t)|^{4}+\sup\limits_{t\in[0,T]}\tilde{\mbf E}|\Phi(t)|^{4}<\infty$, and there is $\alpha>0$ such that $\sup\limits_{N\ge1}\mbf E|\Phi^N(t)-\Phi^N(s)|^2+\tilde{\mbf E}|\Phi(t)-\Phi(s)|^{2}\le K|t-s|^{\alpha}$ for any $t,s\in[0,T]$.
	\end{itemize}
	Then we have $Z^N\overset{d}{\Rightarrow}Z$ in $\mbf C([0,T];\mbb R^{d_2})$ with $Z$ being the strong solution of
	\begin{align}\label{sec2eq1}
		Z(t)=\int_0^tB_0(Y(s))Z(s)\ud s+\sum_{i=1}^m\int_{0}^{t}B_i(Y(s))Z(s)\ud W_i(s)+\Phi(t),~t\in[0,T].
	\end{align}
\end{theo}
\begin{proof}
The proof uses a frozen-coefficient approximation. 	Let $Z^{N,n}=\{Z^{N,n}(t),t\in[0,T]\}$ be the strong solution of
\begin{align*}
	Z^{N,n}(t)
	=\int_0^tB_0(Y(\kappa_n(s)))Z^{N,n}(\kappa_n(s))\ud s+\sum_{i=1}^m\int_0^tB_i(Y(\kappa_n(s)))Z^{N,n}(\kappa_n(s))\ud W_i(s)+\Phi^N(t),
\end{align*}	
where $\kappa_n(s)=\lfloor\frac{sn}{T}\rfloor\frac{T}{n},$ $n\in\mbb N^+$. The BDG inequality, the moment bounds and the Hölder regularity of $Y$ give
$ \sup_{N\ge1}\mbf E\|Z^{N,n}-Z^N\|_{\mbf C([0,T])}^{2}
 \le K n^{-\min(1,\alpha,2\beta)}.$
The same estimate holds for the limiting equation with $\Phi$ in place of $\Phi^N$. For fixed $n$, the frozen equation is a continuous map of $(Y,W_1,\ldots,W_m,\Phi^N)$, so Proposition \ref{Pro1} and the continuous mapping theorem give
$Z^{N,n}\overset{d}{\Longrightarrow}Z^{\infty,n}$. Here,  $Z^{\infty,n}:=\Gamma^n(Y,W_1,\ldots,W_m,\Phi)$ is the strong solution of
\begin{align*}
	Z^{\infty,n}(t)=&\;\int_0^tB_0(Y(\kappa_n(s)))Z^{\infty,n}(\kappa_n(s))\ud s+\sum_{i=1}^m\int_0^t B_i(Y(\kappa_n(s)))Z^{\infty,n}(\kappa_n(s))\ud W_i(s)\\
	&\;+\Phi(t),~t\in[0,T].	
\end{align*}
Letting $n\to\infty$  and applying the approximation theorem for convergence in distribution \cite[Theorem 3.2]{HJWY24} proves the result.
\end{proof}

\subsection{Truncated Brownian increments}
Both applications below involve implicitness in the diffusion terms. To ensure pathwise solvability of the implicit equations, we use truncated Brownian increments. Let $\Delta W_n=\sqrt h\,\xi_n$ denote one component of a Brownian increment, where $\xi_n\sim\mcal N(0,1)$. Define $\Delta\widehat W_n=\sqrt h\,\zeta_{n,h}$ by
\begin{align}\label{zeta}
	\zeta_{n,h}:=\begin{cases}
		\xi_n,\quad &|\xi_n|\le A_h,\\
		A_h,\quad &\xi_n>A_h, \\
		-A_h,\quad &\xi_n<-A_h,
	\end{cases}
\end{align}
where $A_h:=\sqrt{2\kappa|\ln h|}$ with $\kappa\ge 1$ being a constant. Then one has the following properties:
\begin{align}
	\mbf E|\xi_n-\zeta_{n,h}|^{p}&\le Kh^{\kappa},~\forall~p\ge 1, \label{zeta1}\\
	\|\xi_n^m-\zeta_{n,h}^m\|_{\mbf L^p(\Omega)}&\le K(p,m,\varepsilon)h^{\kappa/p-\varepsilon},~\forall ~m\in\mbb N^+,~p\ge1,~\varepsilon\in(0,\frac{\kappa}{p})\label{zeta2}, \\
	|\mbf E(\xi_n^m-\zeta_{n,h}^m)|&\le K(m,\varepsilon)h^{\kappa-\varepsilon},~\forall ~m\in\mbb N^+,~\varepsilon\in(0,\kappa) \label{zeta3}.
\end{align}
In fact, \eqref{zeta1} can be proved similarly to \cite[Page 39, Eq. (3.36)]{Milsteinbook}. Note that $\mbf E|\zeta_{n,h}|^q\le \mbf E|\xi_n|^q\le K(q)$ for any $q\ge 1$. By the H\"older inequality, for any $\varepsilon\in(0,\frac{\kappa}{p})$, we have that
$\mbf E|\xi_n^m-\zeta_{n,h}^m|^p=\mbf E\big(|m\int_0^1(\zeta_{n,h}+\lambda(\xi_n-\zeta_{n,h}))^{m-1}\ud\lambda|^p|\xi_n-\zeta_{n,h}|^p\big)\le \big(\mbf E|\xi_n-\zeta_{n,h}|^{\frac{p}{1-\varepsilon p/\kappa}}\big)^{1-\varepsilon p/\kappa}K(p,\varepsilon,m)\le K(p,\varepsilon,m)h^{\kappa-p\varepsilon}$, which verifies \eqref{zeta2}. In addition, \eqref{zeta3} is a direct consequence of \eqref{zeta2} with $p=1$.
For a multidimensional Brownian motion, the construction is applied componentwise to the independent normalized increments $\xi_n^i$. In the scalar-noise application, the component index is suppressed.

\section[Multidimensional stochastic midpoint method]{First application: asymptotic error distribution of stochastic midpoint method}\label{SecMidpoint}
Here we illustrate Framework \ref{frame} for multidimensional, non-commutative noise. Since this first application is intended to display the argument rather than to seek minimal regularity, the coefficients are assumed to be sufficiently smooth. The initial value is assumed to possess all moments needed in the argument. The truncation parameter in \eqref{zeta} is taken sufficiently large.

\subsection{The stochastic midpoint method}
Consider the $d$-dimensional Stratonovich SDE
\begin{align}\label{multiSDE}
	\ud X(t)=f(X(t))\ud t+\sum_{i=1}^{m}g_i(X(t))\circ\ud W_i(t),
	\qquad X(0)=X_0,
\end{align}
where $W_1,\ldots,W_m$ are independent standard Brownian motions. No commutativity condition is imposed on the diffusion fields. Applying \eqref{zeta} componentwise, we write
$\Delta\widehat W_n^i=\sqrt h\,\zeta_{n,h}^i$. The fully implicit stochastic midpoint method is
\begin{align}\label{midpoint}
	X_{n+1}^{\rm mid}=&\;X_n^{\rm mid}
	+h f\left(\frac{X_n^{\rm mid}+X_{n+1}^{\rm mid}}2\right)
	+\sum_{i=1}^{m}g_i\left(\frac{X_n^{\rm mid}+X_{n+1}^{\rm mid}}2\right)
	\Delta\widehat W_n^i ,
\end{align}
with $X_0^{\rm mid}=X_0$. Under the preceding convention, the implicit equation is uniquely solvable for all sufficiently small $h$, uniformly in $n$ and $\omega$. The standard moment estimate and the strong order result for general multidimensional noise give, for every admissible $p\ge2$,
\begin{align}\label{midknown}
	\sup_{0\le n\le N}\mbf E|X_n^{\rm mid}|^p\le K(p,T),\qquad
	\max_{0\le n\le N}\|X_n^{\rm mid}-X(t_n)\|_{\mbf L^p(\Omega)}
	\le K(p,T)h^{1/2};
\end{align}
see, for example, \cite{Burrage98}. Truncating the Brownian increments does not change the order, in view of \eqref{zeta1}--\eqref{zeta3}. Thus Step 1 of Framework \ref{frame} is already available, and it remains to construct a suitable appurtenant method.

\subsection{An appurtenant method}
Let $X_{t,x}^{\rm mid}(t+h)$ denote one step of \eqref{midpoint} starting from $x$ at time $t$, and write $\Delta\widehat W_i$ for the corresponding truncated Brownian increment. The following expansion  follows from a standard argument.

\begin{lem}\label{midexpansion}
	For every admissible $p\ge2$, the one-step approximation of \eqref{midpoint} admits the representation
	\begin{align}\label{midexpansioneq}
		X_{t,x}^{\rm mid}(t+h)=&\;x+h f(x)
		+\sum_{i=1}^m g_i(x)\Delta\widehat W_i+\frac12\sum_{i,j=1}^m\nabla g_i(x)g_j(x)
		\Delta\widehat W_i\Delta\widehat W_j
		+R_{t,x}^{\rm mid},
	\end{align}
	where, for some $\alpha\ge1$,
		$|\mbf E R_{t,x}^{\rm mid}|
		\le K(1+|x|^\alpha)h^2,~
		\|R_{t,x}^{\rm mid}\|_{\mbf L^p(\Omega)}
		\le K(p)(1+|x|^\alpha)h^{3/2}.$
\end{lem}

The principal part in \eqref{midexpansioneq} suggests the following candidate. With
$\Delta W_i=W_i(t+h)-W_i(t)$, define
\begin{align}\label{midappcandidate}
	\widetilde X_{t,x}(t+h)=x+h f(x)+\sum_{i=1}^m g_i(x)\Delta W_i
	+\frac12\sum_{i,j=1}^m\nabla g_i(x)g_j(x)
	\Delta W_i\Delta W_j.
\end{align}
The next estimate plays the same role as Lemma \ref{YSRK-Ybar}.

\begin{lem}\label{midapplocal}
	For every admissible $p\ge2$, there is $\alpha\ge1$ such that
	\begin{align}\label{midapplocalbound}
		\big|\mbf E\big(X_{t,x}^{\rm mid}(t+h)-\widetilde X_{t,x}(t+h)\big)\big|
		&\le K(1+|x|^\alpha)h^2,\notag\\
		\|X_{t,x}^{\rm mid}(t+h)-\widetilde X_{t,x}(t+h)\|_{\mbf L^p(\Omega)}
		&\le K(p)(1+|x|^\alpha)h^{3/2}.
	\end{align}
\end{lem}
This follows directly from Lemma \ref{midexpansion} and the truncation estimates \eqref{zeta1}--\eqref{zeta3}, applied componentwise.

The products in \eqref{midappcandidate} are not well suited to an adapted continuization. For $i,j=1,\ldots,m$, the It\^o product formula gives
\begin{align}\label{midproduct}
	\Delta W_i\Delta W_j
	=\delta_{ij}h
	+\int_t^{t+h}\big(W_i(s)-W_i(t)\big)\ud W_j(s)
	+\int_t^{t+h}\big(W_j(s)-W_j(t)\big)\ud W_i(s).
\end{align}
Here $\delta_{ij}$ denotes the Kronecker delta.
Writing	$\bar f=f+\frac12\sum_{i=1}^m\nabla g_i g_i,$
and substituting \eqref{midproduct} into \eqref{midappcandidate}, we obtain the equivalent adapted representation
\begin{align}\label{midappstep}
&\;	\widetilde X_{t,x}(t+h)=x+h\bar f(x)+\sum_{j=1}^m g_j(x)\Delta W_j\notag\\
	&\;+\frac12\sum_{i,j=1}^m
	\big(\nabla g_i(x)g_j(x)+\nabla g_j(x)g_i(x)\big)
	\int_t^{t+h}\big(W_i(s)-W_i(t)\big)\ud W_j(s).
\end{align}
Starting from $\widetilde X_0=X_0$, define
$\widetilde X_{n+1}=\widetilde X_{t_n,\widetilde X_n}(t_{n+1})$ for $n=0,\ldots,N-1$.

\begin{lem}\label{midapp}
	For every admissible $p\ge2$,
	\begin{align}\label{midappbound}
		\max_{0\le n\le N}\|X_n^{\rm mid}-\widetilde X_n\|_{\mbf L^p(\Omega)}
		&\le K(p,T)h,\notag\\
		\max_{0\le n\le N}\big(
		\|X_n^{\rm mid}-X(t_n)\|_{\mbf L^p(\Omega)}
		+\|\widetilde X_n-X(t_n)\|_{\mbf L^p(\Omega)}
		\big)&\le K(p,T)h^{1/2}.
	\end{align}
	Consequently, $\{\widetilde X_n\}_{n=0}^N$ is an appurtenant method of \eqref{midpoint}.
\end{lem}
\begin{proof}
Indeed, Lemma \ref{midapplocal} verifies condition (A1) of Theorem \ref{fundamental} with $p_1=2$ and $p_2=3/2$. Under the standing smoothness convention, conditions (A2) and (A3) follow, respectively, from the usual discrete moment argument and the Lipschitz estimates for the coefficients in \eqref{midappstep}, together with the It\^o isometry. The first estimate in \eqref{midappbound} then follows from Theorem \ref{fundamental}, and the second follows from \eqref{midknown} and the triangle inequality.
\end{proof}
\subsection{Asymptotic error distribution}
We now carry out Step 3 of Framework \ref{frame}. Let $\kappa_N(t)=t_n$ for $t\in[t_n,t_{n+1})$. The form \eqref{midappstep} gives the following adapted continuization of $\{\widetilde X_n\}_{n=0}^N$:
\begin{align}\label{midcontinuization}
	\widetilde X(t)=&\;X_0+\int_0^t\bar f\big(\widetilde X(\kappa_N(s))\big)\ud s
	+\sum_{j=1}^m\int_0^t g_j\big(\widetilde X(\kappa_N(s))\big)\ud W_j(s)\notag\\
	&+\frac12\sum_{i,j=1}^m\int_0^t
	\big(\nabla g_i\,g_j+\nabla g_j\,g_i\big)
	\big(\widetilde X(\kappa_N(s))\big)
	\big(W_i(s)-W_i(\kappa_N(s))\big)\ud W_j(s).
\end{align}
In particular, $\widetilde X(t_n)=\widetilde X_n$. Notice that \eqref{midcontinuization} contains only adapted stochastic integrals and no products of complete Wiener increments.

We use the convention $[g_i,g_j]=\nabla g_j\,g_i-\nabla g_i\,g_j$ 
for the Lie bracket of two diffusion fields, and define the continuous local martingale
\begin{align}\label{midPsi}
	\Psi^N(t)=&-\frac{\sqrt N}{2}\sum_{1\le i<j\le m}
	\int_0^t[g_i,g_j]\big(\widetilde X(\kappa_N(s))\big)
	\Big\{\big(W_i(s)-W_i(\kappa_N(s))\big)\ud W_j(s)\notag\\
	&\hspace{58mm}
	-\big(W_j(s)-W_j(\kappa_N(s))\big)\ud W_i(s)\Big\}.
\end{align}

\begin{lem}\label{miderrorrep}
Denote $V^N:=\sqrt N(\widetilde X-X)$. Then $V^N$ has the following representation
	\begin{align}\label{midVN}
		V^N(t)=&\int_0^t\nabla\bar f(X(s))V^N(s)\ud s
		+\sum_{i=1}^m\int_0^t\nabla g_i(X(s))V^N(s)\ud W_i(s)
		+\Psi^N(t)+R^N(t),
	\end{align}
	where
		$\sup_{t\in[0,T]}\mbf E|R^N(t)|^2\le Kh$ and
		$\sup_{N\ge1}\sup_{t\in[0,T]}\mbf E|V^N(t)|^2<\infty.$
\end{lem}
\noindent The proof of Lemma \eqref{midVN} follows from a direct computation.

To separate the negligible remainder from the driving martingale, let $\widehat V^N$ be the unique strong solution of
\begin{align}\label{midauxiliary}
	\widehat V^N(t)=&\int_0^t\nabla\bar f(X(s))\widehat V^N(s)\ud s
	+\sum_{i=1}^m\int_0^t\nabla g_i(X(s))\widehat V^N(s)\ud W_i(s)
	+\Psi^N(t).
\end{align}
Subtracting \eqref{midauxiliary} from \eqref{midVN}, and using Lemma \ref{miderrorrep}, the It\^o isometry, and Gronwall's inequality, gives
\begin{align}\label{midauxiliaryerror}
	\sup_{t\in[0,T]}\mbf E|V^N(t)-\widehat V^N(t)|^2\le Kh\longrightarrow0.
\end{align}

The stable limit of $\Psi^N$ is the only point which remains to be considered, which can be solved by applying \cite[Theorem 4-1]{Jacod97}.

\begin{lem}\label{midJacod}
	The sequence $\Psi^N$ converges stably to $\Psi$ in
	$\mbf C([0,T];\mbb R^d)$, where
	\begin{align}\label{midPsilimit}
		\Psi(t)=\frac{\sqrt T}{2}\sum_{1\le i<j\le m}
		\int_0^t[g_i,g_j](X(s))\ud\widetilde W_{ij}(s).
	\end{align}
	Here $\widetilde W_{ij}$, $1\le i<j\le m$, are independent standard Brownian motions, independent of $\mbf W$.
\end{lem}
\begin{proof}
	The proof is based on \cite[Theorem 4-1]{Jacod97}.  For $l=1,\ldots,m$, the cross variation between $\Psi^{N,k}$ and $W_l$ is
	\begin{align*}
		\LL\Psi^{N,k},W_l\RR_t
		=&-\frac{\sqrt N}{2}\sum_{i<j}\int_0^t
		[g_i,g_j]^k\big(\widetilde X(\kappa_N(s))\big)\\
		&\quad\times\big\{
		\delta_{jl}\big(W_i(s)-W_i(\kappa_N(s))\big)
		-\delta_{il}\big(W_j(s)-W_j(\kappa_N(s))\big)
		\big\}\ud s.
	\end{align*}
	The increments on different time intervals are orthogonal after conditioning on the left endpoints. The It\^o isometry therefore gives
	\begin{align}\label{midcrossW}
		\mbf E\big|\LL\Psi^{N,k},W_l\RR_t\big|^2
		\le Kh\longrightarrow0,
		\qquad k=1,\ldots,d,\quad l=1,\ldots,m.
	\end{align}

	We now compute the quadratic covariations of the components of $\Psi^N$. From \eqref{midPsi},
	\begin{align}\label{midselfpre}
		\LL\Psi^{N,k},\Psi^{N,l}\RR_t
		=&\frac N4\sum_{i<j}\int_0^t
		[g_i,g_j]^k\big(\widetilde X(\kappa_N(s))\big)
		[g_i,g_j]^l\big(\widetilde X(\kappa_N(s))\big)\notag\\
		&\quad\times\left\{
		|W_i(s)-W_i(\kappa_N(s))|^2
		+|W_j(s)-W_j(\kappa_N(s))|^2
		\right\}\ud s+D_{kl}^N(t).
	\end{align}
	To make the last term explicit, the differential of the $k$th component of $\Psi^N$ may first be written as
	\begin{align*}
		\ud\Psi^{N,k}(s)
		=\frac{\sqrt N}{2}\sum_{i=1}^m\sum_{\substack{j=1\\j\ne i}}^m
		[g_i,g_j]^k\big(\widetilde X(\kappa_N(s))\big)
		\big(W_j(s)-W_j(\kappa_N(s))\big)\ud W_i(s).
	\end{align*}
	Consequently, the off-diagonal terms in the product of the two integrands give
	\begin{align}\label{midDexplicit}
		D_{kl}^N(t)
		={}&\frac N4\sum_{i=1}^m
		\sum_{\substack{1\le j_1<j_2\le m\\j_1,j_2\ne i}}
		\int_0^t
		\Big\{
		[g_i,g_{j_1}]^k\big(\widetilde X(\kappa_N(s))\big)
		[g_i,g_{j_2}]^l\big(\widetilde X(\kappa_N(s))\big)\notag\\
		&\hspace{30mm}+
		[g_i,g_{j_2}]^k\big(\widetilde X(\kappa_N(s))\big)
		[g_i,g_{j_1}]^l\big(\widetilde X(\kappa_N(s))\big)
		\Big\}\notag\\
		&\hspace{18mm}\times
		\big(W_{j_1}(s)-W_{j_1}(\kappa_N(s))\big)
		\big(W_{j_2}(s)-W_{j_2}(\kappa_N(s))\big)\ud s.
	\end{align}
	A standard conditional-expectation argument yields
		$\sup_{t\in[0,T]}\mbf E|D_{kl}^N(t)|^2\le Kh.$

	For any bounded Lipschitz function $H$ and $i=1,\ldots,m$, we have
	\begin{align}\label{midaverage}
		N\int_0^tH\big(\widetilde X(\kappa_N(s))\big)
		|W_i(s)-W_i(\kappa_N(s))|^2\ud s
		\rightarrow
		\frac T2\int_0^tH(X(s))\ud s
		\quad\text{in probability}.
	\end{align}
In fact, by It\^o's formula, for every $s\in[0,T]$,
	\begin{align}\label{midItoincrement}
		|W_i(s)-W_i(\kappa_N(s))|^2
		={}&2\int_{\kappa_N(s)}^s
		\big(W_i(r)-W_i(\kappa_N(s))\big)\ud W_i(r)
		+s-\kappa_N(s).
	\end{align}
	The contribution of the stochastic integral in \eqref{midItoincrement} is
	\begin{align*}
		R_{H,i}^N(t)={}&2N\int_0^t
		H\big(\widetilde X(\kappa_N(s))\big)
		\int_{\kappa_N(s)}^s
		\big(W_i(r)-W_i(\kappa_N(s))\big)\ud W_i(r)\ud s.
	\end{align*}
A conditional-expectation argument gives 
		$\sup_{t\in[0,T]}\mbf E|R_{H,i}^N(t)|^2
		\le Kh.$

Moreover, we  apply \cite[Proposition 4.2]{HJWY24} to get
	\begin{align*}
		N\int_0^tH\big(\widetilde X(\kappa_N(s))\big)
		(s-\kappa_N(s))\ud s
		\longrightarrow\frac T2\int_0^tH(X(s))\ud s
	\end{align*}
	in $\mbf L^2(\Omega)$. Together with the preceding estimate for $R_{H,i}^N(t)$, this proves \eqref{midaverage}. Combining \eqref{midselfpre}--\eqref{midaverage}, we arrive at
	\begin{align}\label{midselflimit}
		\LL\Psi^{N,k},\Psi^{N,l}\RR_t
		\longrightarrow\frac T4\sum_{i<j}\int_0^t
		[g_i,g_j]^k(X(s))[g_i,g_j]^l(X(s))\ud s
	\end{align}
	in probability for every $t\in[0,T]$ and $k,l=1,\ldots,d$.

Applying \cite[Theorem 4-1]{Jacod97}, we obtain
	$\Psi^N\overset{\mathrm{stably}}{\Longrightarrow}\Psi$ in
	$\mbf C([0,T];\mbb R^d)$, where $\Psi$ is a continuous conditionally Gaussian martingale such that
	\begin{align*}
		\LL\Psi^k,W_i\RR_t=0,\quad
		\LL\Psi^k,\Psi^l\RR_t
		=\frac T4\sum_{i<j}\int_0^t
		[g_i,g_j]^k(X(s))[g_i,g_j]^l(X(s))\ud s.
	\end{align*}
	Further, by \cite[Proposition 1-4]{Jacod97}, $\Psi$ can be represented on an extension of the original probability space as
		$\Psi^k(t)=\sum_{i=1}^m\int_0^t u_i^k(s)\ud W_i(s)
		+\sum_{1\le i<j\le m}\int_0^t v_{ij}^k(s)\ud\widetilde W_{ij}(s),$
	where $\{\widetilde W_{ij}:1\le i<j\le m\}$ are independent standard Brownian motions, independent of $\mbf W$. The first covariation identity gives $u_i^k(s)=0$, whereas the second gives
	\begin{align*}
		\sum_{i<j}v_{ij}^k(s)v_{ij}^l(s)
		=\frac T4\sum_{i<j}[g_i,g_j]^k(X(s))[g_i,g_j]^l(X(s)).
	\end{align*}
	Thus, we may take
	$v_{ij}^k(s)=\frac{\sqrt T}{2}[g_i,g_j]^k(X(s))$, which yields \eqref{midPsilimit}. Thus, the proof is complete.
\end{proof}

\begin{theo}\label{midpointlimit}
	It holds that
$	\sqrt N\big(X_N^{\rm mid}-X(T)\big)
		\overset{d}{\Longrightarrow}V(T),$
	where $V$ is the unique strong solution of
	\begin{align}\label{midlimit}
		V(t)=&\int_0^t\nabla\bar f(X(s))V(s)\ud s
		+\sum_{i=1}^m\int_0^t\nabla g_i(X(s))V(s)\ud W_i(s)\notag\\
		&+\frac{\sqrt T}{2}\sum_{1\le i<j\le m}
		\int_0^t[g_i,g_j](X(s))\ud\widetilde W_{ij}(s).
	\end{align}
\end{theo}
\begin{proof}
Lemma \ref{midJacod} and Theorem \ref{stablecon} yield $\widehat V^N(T)\overset{d}{\Longrightarrow}V(T)$.  It follows from \eqref{midauxiliaryerror} and Slutsky's theorem that $V^N(T)\overset{d}{\Longrightarrow}V(T)$. Finally, Lemma \ref{midapp} implies
		$\sqrt N\|X_N^{\rm mid}-\widetilde X_N\|_{\mbf L^2(\Omega)}
		\le K\sqrt N h\rightarrow0.$
	A second application of Slutsky's theorem completes the proof.
\end{proof}

\begin{rem}
	If the diffusion fields commute, then $[g_i,g_j]=0$ for every $i,j$, and the limit in \eqref{midlimit} vanishes. This does not mean that the midpoint method has no asymptotic error. Rather, the leading L\'evy-area defect is absent and the error has to be examined at the finer scale $N$. The scalar-noise analysis in Section \ref{SecAsymptoticSRK} is an instance of this situation.
\end{rem}

\section{Properties of stochastic Runge--Kutta method}\label{SecSRK}
In this section, for SDEs with scalar multiplicative noise, we establish the  strong convergence of the SRK method and construct an appurtenant method that shares its asymptotic error distribution. The first-order convergence conditions for implicit SRK methods in \cite{Birrage04} are formulated in the root-mean-square sense. Since the subsequent asymptotic analysis requires higher-moment $L^{2p}(\Omega)$-estimates with $p>1$ for the scheme driven by the truncated Brownian increments defined in \eqref{zeta}, we include the corresponding strong-convergence analysis for completeness.

Consider the following SDE with scalar noise:
\begin{align}\label{mulSDE1}
	\begin{cases}
		\ud Y(t)=f(Y(t))\ud t+g(Y(t))\circ\ud W(t),~ t\in(0,T], \\
		Y(0)=Y_0\in\mbb R^d,
	\end{cases}
\end{align}
where $f,g:\mbb R^d\to\mbb R^d$ are Lipschitz continuous, $W$ is a one-dimensional standard Brownian motion, and $g(Y(t))\circ\ud W(t)$ denotes the Stratonovich stochastic differential.
When $g\in\mbf C^1(\mbb R^d)$, \eqref{mulSDE1} is equivalent to the following It\^o SDE:
\begin{align}\label{mulSDE2}
	\begin{cases}
		\ud Y(t)=\bar{f}(Y(t))\ud t+g(Y(t))\ud W(t),~ t\in(0,T], \\
		Y(0)=Y_0\in\mbb R^d,
	\end{cases}
\end{align}
where $\bar{f}(y)=f(y)+\frac{1}{2}\nabla g(y)g(y)$, $y\in\mbb R^d$, and  $g(Y(t))\ud W(t)$ denotes the It\^o stochastic differential. For convenience, we always assume that $Y_0$ is nonrandom.
\begin{assum}\label{assum1}
	Assume that $f$, $g$ and $(\nabla g)g$ are Lipschitz continuous.
\end{assum}
It is well-known that under Assumption \ref{assum1}, \eqref{mulSDE1} admits a unique strong solution which satisfies that for any $p\ge 2$, and $t,s\in[0,T]$,
\begin{align}
	\|Y(t)\|_{\mbf L^p(\Omega)}&\le C_1(p)e^{C_1(p)T}(1+|Y_0|),\label{Ybound}\\
	\|Y(t)-Y(s)\|_{\mbf L^p(\Omega)}&\le K(p,T)|t-s|^{1/2},\label{Yt-Ys}
\end{align}
where $C_1(p)$ depends on $p$ but is independent of $T$; see \cite[Chapter 2.4]{Maoxuerong}.

\subsection{Stochastic Runge--Kutta method}
Consider the following $s_0$-stage SRK for \eqref{mulSDE1}
\begin{align}\label{SRK1}
	\begin{cases}
		Z_{n,i}=Y_n^{SRK}+h\sum\limits_{j=1}^{s_0}a_{ij}f(Z_{n,j})+\DW_n\sum\limits_{j=1}^{s_0}b_{ij}g(Z_{n,j}),~i=1,\ldots,s_0,\\
		Y^{SRK}_{n+1}=Y^{SRK}_{n}+h\sum\limits_{i=1}^{s_0}\alpha_if(Z_{n,i})+\DW_n\sum\limits_{i=1}^{s_0}\beta_ig(Z_{n,i}),~n=0,\ldots,N-1,
	\end{cases}
\end{align}
with $Y_0^{SRK}=Y_0$. The coefficients of \eqref{SRK1} can be represented by the Butcher tableau of the form ~
\renewcommand{\arraystretch}{1.5}
\begin{tabular}{|c| c}
	\( A \) & \( B \) \\
	\hline
	\(\alpha^\top\) & \(\beta^\top\)  \\
\end{tabular},
where $A=(a_{ij})$, $B=(b_{ij})$, $\alpha=(\alpha_1,\ldots,\alpha_{s_0})^\top$ and $\beta=(\beta_1,\ldots,\beta_{s_0})^\top$.
Then the one-step approximation of \eqref{SRK1} can be written as
\begin{align}\label{SRK2}
	Y^{SRK}_{t,y}(t+h):=y+h\sum\limits_{i=1}^{s_0}\alpha_if(Z_i)+\DW\sum\limits_{i=1}^{s_0}\beta_ig(Z_i).
\end{align}
Here $\DW=\sqrt{h}\zeta$ is the truncation of $W(t+h)-W(t)=\sqrt{h}\xi$ for which $\xi\sim\mcal N(0,1)$ and $\zeta$ is defined as in \eqref{zeta} with $\xi_n$ replaced by $\xi$. In addition, 
$Z=(Z_1^\top,\ldots,Z_{s_0}^\top)^\top$ is determined by
\begin{align}\label{Z}
	Z=(e\otimes I_d)y+h(A\otimes I_d)F(Z)+\DW(B\otimes I_d)G(Z),
\end{align}
where $F(Z)=(f(Z_1)^\top,\ldots,f(Z_{s_0})^\top)^\top$, $G(Z)=(g(Z_1)^\top,\ldots,g(Z_{s_0})^\top)^\top$, every component of $e\in\mbb R^{s_0}$ equals to $1$, $I_d\in\mbb R^{d\times d}$ is the identity  matrix, and $\otimes$ denotes the Kronecker product of matrices. The following lemma gives the solvability of \eqref{Z}, which can be established by the contraction mapping principle.

\begin{lem}\label{solvability}
	Let Assumption \ref{assum1} hold. Then there is $h_1>0$ such that for any $h\in(0,h_1]$, $y\in\mbb R^d$, $\omega\in\Omega$, \eqref{Z} is uniquely solvable. Moreover, there is $C_0>0$ such that 
	\begin{align*}
		|Z_i-y|\le C_0(1+|y|)(h+|\DW|),~h\in(0,h_1].
	\end{align*}
\end{lem}

Since $h_1$ in Lemma \ref{solvability} is independent of $y$ and $\omega$, we obtain the well-posedness of \eqref{SRK1}. Next, we present the moment boundedness of \eqref{SRK1}.

\begin{lem}\label{SRK1bound}
	Let Assumption \ref{assum1} hold. Then for any $h\le h_1\wedge 1$, we have that for any $p\ge1$,
$\sup_{n=0,1,\ldots,N}\mbf E|Y^{SRK}_n|^p \le K(p,T). $
\end{lem}
The proof of Lemma \ref{SRK1bound} is mainly based on \cite[Page 102, Lemma 2.2]{Milsteinbook}, and we omit it.
In the subsequent text of this section, we always assume that $h\le h_1\wedge 1$ without extra statement.  The first-order root-mean-square convergence conditions of \eqref{SRK1} have been obtained in \cite[Eq. (3.7)]{Birrage04}. Here, we generalize them to  the $\|\cdot\|_{L^{2p}(\Omega)}$-norm for $p\in[1,\frac{\kappa}{2}]$.
\begin{theo}\label{SRK1converge}
	Let Assumption \ref{assum1} hold. Let $f\in\mbf C^2(\mbb R^d)$, $g\in\mbf C^3(\mbb R^d)$ and $\mcal D^{2}f,\,\mcal D^3g\in \mbf F$. If
	$\alpha^\top e=\beta^\top e=1$, $\beta^\top(Be)=\frac{1}{2}$ and $\kappa\ge 2$, then for any $p\in[1,\frac{\kappa}{2}]$,
		$\sup_{n=0,1,\ldots,N}\|Y^{SRK}_n-Y(t_n)\|_{\mbf L^{2p}(\Omega)}\le Kh.$
\end{theo}
\noindent\emph{Proof} Taylor expansion of the stage equations gives the local mean error $O(h^2)$ and the $L^{2p}$ local error $O(h^{3/2})$. The conditions of Theorem \ref{fundamental} follow from Lemma \ref{solvability}, the moment bound, and \eqref{zeta1}--\eqref{zeta3}; its conclusion yields the stated global order.

\subsection{An appurtenant numerical method}
For $a=(a_1,\ldots,a_d)^\top,~b=(b_1,\ldots,b_d)^\top\in\mbb R^d$, denote $a*b\in\mbb R^d$ with $(a*b)_i=a_ib_i$, $i=1,\ldots,d$ and $a^m\in\mbb R^d$ with $(a^m)_i=a_i^m$. The following lemma gives a sufficient expansion of $Y_{t,y}^{SRK}(t+h)$ and identifies the terms needed for the appurtenant method. It follows from Taylor expansions of the stage equations: terms of strong order at most $2$ are collected in $F_1,F_2,F_3,F_4$, while the remainder has order $5/2$ in $L^p$.
\begin{lem}\label{YSRKexpan}
	Let Assumption \ref{assum1} hold. Let $f\in\mbf C^3(\mbb R^d)$, $g\in\mbf C^4(\mbb R^d)$ and $\mcal D^{3}f,\,\mcal D^4g\in \mbf F$. Moreover, let
	$\alpha^\top e=\beta^\top e=1$, $\beta^\top(Be)=\frac{1}{2}$. Then $Y^{SRK}_{t,y}(t+h)$ has following representation
	\begin{align}
	&\;	Y^{SRK}_{t,y}(t+h)=y+\DW g(y)+hf(y)+\frac{1}{2}\DW^2\nabla g(y)g(y)+\DW hF_1(y)+\DW^3 F_2(y)\notag\\
		&\;+h^2\alpha^\top(Ae)\nabla f(y)f(y)+\DW^2hF_3(y)+\DW^4F_4(y)+R^2_{Y^{SRK}},\label{YSRKexpansion}
	\end{align}
	where
	\begin{align*}
		F_1(y)=&\; \alpha^\top(Be)\nabla f(y)g(y)+\beta^\top(Ae)\nabla g(y)f(y),\\
		F_2(y)=&\;\beta^\top(B(Be))(\nabla g(y))^2g(y)+\frac{1}{2}\beta^\top(Be)^2\mcal D^2g(y)(g(y),g(y)),\\
		F_3(y)=&\;\alpha^\top(B(Be))\nabla f(y)\nabla g(y)g(y)+\beta^\top(A(Be))\nabla g(y)\nabla f(y)g(y)+\beta^\top(B(Ae))(\nabla g(y))^2f(y)\\
		&\;+\frac{1}{2}\alpha^\top(Be)^2\mcal D^2f(y)(g(y),g(y))+\beta^\top((Ae)*(Be))\mcal D^2g(y)(f(y),g(y)),\\
		F_4(y)=&\;\beta^\top(B(B(Be)))(\nabla g(y))^3g(y)+\frac{1}{2}\beta^\top(B(Be)^2)\nabla g(y)\mcal D^2g(y)(g(y),g(y))\\
		&\;+\beta^\top((Be)*(B(Be)))\mcal D^2g(y)(g(y),\nabla g(y)g(y))+\frac{1}{6}\beta^\top(Be)^3\mcal D^3g(y)(g(y),g(y),g(y)),
	\end{align*}
	and there is  $\alpha'\ge 1$ such that for any $p\ge1$, $\|R^2_{Y^{SRK}}\|_{\mbf L^p(\Omega)}\le K(p)(1+|y|^{\alpha'})h^{5/2}$.  
\end{lem}

According to \eqref{YSRKexpansion}, we can define the candidate of the appurtenant method of the SRK method \eqref{SRK1}. For this end, let $\bar{Y}_{t,y}(t+h)$ be the one-step approximation defined by
\begin{align}\label{Ybar}
	\bar{Y}_{t,y}(t+h)=&\;y+\Delta W g(y)+hf(y)+\frac{1}{2}\Delta W^2\nabla g(y)g(y)+\Delta W hF_1(y)+\Delta W^3 F_2(y)\notag\\
	&\;+h^2\alpha^\top(Ae)\nabla f(y)f(y)+\Delta W^2hF_3(y)+\Delta W^4F_4(y).
\end{align}
Then, we have following properties.
\begin{lem}\label{YSRK-Ybar}
	Let the assumptions of Lemma \ref{YSRKexpan} hold. If $\kappa\ge 3$, then there is some $\alpha''\ge 1$ such that for $0\le t\le T-h$ and $y\in\mbb R^d$,
	\begin{align*}
		\big|\mbf E(Y^{SRK}_{t,y}(t+h)-\bar{Y}_{t,y}(t+h))\big|&\le K(1+|y|^{\alpha''})h^{5/2}, \\
		(\mbf E|Y^{SRK}_{t,y}(t+h)-\bar{Y}_{t,y}(t+h)|^2)^{1/2}&\le K(1+|y|^{\alpha''})h^2.
	\end{align*}
\end{lem}
\begin{proof}
	By \eqref{YSRKexpansion} and \eqref{Ybar}, one has
	\begin{align*}
		&\;Y^{SRK}_{t,y}(t+h)-\bar{Y}_{t,y}(t+h)\\
		=&\;(\DW-\Delta W)g(y)+\frac{1}{2}(\DW^2-\Delta W^2)\nabla g(y)g(y)+(\DW-\Delta W) hF_1(y)\\
		&\;+(\DW^3-\Delta W^3) F_2(y)+(\DW^2-\Delta W^2)hF_3(y)+(\DW^4-\Delta W^4)F_4(y)+R^2_{Y^{SRK}}.
	\end{align*}
	Using \eqref{zeta3} gives $|\mbf E(\DW^m-\Delta W^m)|\le Kh^{\frac{m}{2}+\kappa-\varepsilon}\le Kh^{4-\varepsilon}$ for $m\ge 2$ and $\kappa\ge 3$. Applying \eqref{zeta1}, we obtain $(\mbf E(\DW-\Delta W)^2)^{1/2}\le Kh^{\frac{\kappa+1}{2}}\le Kh^2$. Further, it follows from \eqref{zeta2} that  $(\mbf E(\DW^m-\Delta W^m)^2)^{1/2}\le Kh^{\frac{m+\kappa}{2}-\varepsilon}\le Kh^{5/2-\varepsilon}$ for $m\ge 2$.
	Combining the above estimates, $\|R^2_{Y^{SRK}}\|_{\mbf L^p(\Omega)}\le K(p)(1+|y|^{\alpha'})h^{5/2}$ (see Lemma \ref{YSRKexpan}), and $F_i\in\mbf F$ for $i=1,2,3,4$, we  complete the proof.
\end{proof}

Let $\{\bar{Y}_n\}_{n=0}^N$ be the numerical solution  generated by $\bar Y_{t,y}(t+h)$, i.e., $\bar{Y}_{n+1}=\bar Y_{t_n,\bar Y_n}(t_{n+1})$, $n=0,\ldots,N-1$ and $\bar{Y}_0=Y_0$.
Based on Lemma \ref{YSRK-Ybar} and Theorem \ref{fundamental}, it can be shown that the root-mean-square error order between $\bar{Y}_n-Y^{SRK}_n$ is $1.5$, and $\{\bar{Y}_n\}_{n=0}^N$ has root-mean-square order $1$ when approximating \eqref{mulSDE1}. This is to say, $\{\bar{Y}_n\}_{n=0}^N$ is a candidate of  the appurtenant method of the SRK method \eqref{SRK1}.
However, the terms $\Delta W^{i}$, $i=2,3,4$ in \eqref{Ybar} are very inconvenient for us to give a proper continuous version of $\{\bar{Y}_n\}_{n=0}^N$ for studying its asymptotic error distribution. For this end, we use the multiple It\^o integral to represent $\Delta W^{i}$, $i=2,3,4$, and give the  appurtenant method of \eqref{SRK1}.

First, it holds that
\begin{align}\label{DW2}
	\Delta W^2=(W(t+h)-W(t))^2=2\int_t^{t+h}(W(r)-W(t))\ud W(r)+h.
\end{align}
Then by the It\^o formula, 
\begin{align*}
	(W(s)-W(t))^3=3\int_t^s(W(r)-W(t))^2\ud W(r)+3\int_t^s(W(r)-W(t))\ud r,
\end{align*} 
which, along with \eqref{DW2}, yields
\begin{align*}
	(W(s)-W(t))^3=6\int_t^s\int_t^r(W(u)-W(t))\ud W(u)\ud W(r)+3(s-t)(W(s)-W(t)).	
\end{align*} 
Accordingly,
\begin{align}\label{DW3}
	\Delta W^3=6\int_t^{t+h}\int_t^r(W(u)-W(t))\ud W(u)\ud W(r)+3h\Delta W.
\end{align}
Similarly, applying the It\^o formula and \eqref{DW2}, one has
\begin{align}
	\Delta W^4=4\int_t^{t+h}(W(s)-W(t))^3\ud W(s)+12\int_t^{t+h}\int_t^s(W(r)-W(t))\ud W(r)\ud s+3h^2. \label{DW4}
\end{align}
Plugging \eqref{DW2}-\eqref{DW4} into \eqref{Ybar} leads to
\begin{align}
	&\;\bar{Y}_{t,y}(t+h)=y+\Delta Wg(y)+h\bar f(y)+\nabla g(y)g(y)\int_t^{t+h}(W(s)-W(t))\ud W(s) \notag\\
	&\;+\Delta Wh(F_1(y)+3F_2(y))+6 F_2(y)\int_t^{t+h}\int_t^s(W(r)-W(t))\ud W(r)\ud W(s)\notag\\
	&\;+h^2[\alpha^\top(Ae)\nabla f(y)f(y)+F_3(y)+3F_4(y)]+R_{\bar Y},\label{Ybar2}
\end{align}
where $\bar{f}$ has been defined in \eqref{mulSDE2}  and
\begin{align*}
	R_{\bar Y}:=&\;2h\int_t^{t+h}(W(s)-W(t))\ud W(s) F_3(y)+F_4(y)\Big[4\int_t^{t+h}(W(s)-W(t))^3\ud W(s)\\
	&\;+12\int_t^{t+h}\int_t^s(W(r)-W(t))\ud W(r)\ud s\Big].
\end{align*}

Ignoring the remainder $R_{\bar Y}$, we introduce the one-step approximation  $\widetilde Y_{t,y}(t+h)$:
\begin{align}
	&\;\widetilde{Y}_{t,y}(t+h)=y+\Delta Wg(y)+h\bar f(y)+\nabla g(y)g(y)\int_t^{t+h}(W(s)-W(t))\ud W(s) \notag\\
	&\;+\Delta Wh(F_1(y)+3F_2(y))+6 F_2(y)\int_t^{t+h}\int_t^s(W(r)-W(t))\ud W(r)\ud W(s)\notag\\
	&\;+h^2[\alpha^\top(Ae)\nabla f(y)f(y)+F_3(y)+3F_4(y)]. \label{Ytilde}
\end{align}
Further, define the numerical solution $\{\widetilde{Y}_{n}\}_{n=0}^N$:
\begin{align}\label{Ytilden}
	\widetilde Y_{n+1}=\widetilde{Y}_{t_n,\widetilde Y_n}(t_{n+1}),~n=0,\ldots,N-1,~\widetilde Y_0=Y_0.
\end{align} 
Next we show that \eqref{Ytilden} is an appurtenant method of the SRK method \eqref{SRK1}.
Denote $\mcal H=\{f,g,(\nabla f)f,(\nabla f)g,(\nabla g)f,(\nabla g)g,(\nabla g)^2f,(\nabla g)^2g,(\nabla f\nabla g)g,(\nabla g\nabla f)g, \mcal D^2f(g,g),\\ \mcal D^2g(g,g),
\mcal D^2g(f,g),(\nabla g)^3g,\nabla g\big(\mcal D^2g(g,g)\big),\mcal D^2g(g,(\nabla g)g), \mcal D^3g(g,g,g)\}$.

\begin{lem}\label{YSRKn-Ytilden}
	Let the assumptions of Lemma \ref{YSRKexpan} hold and $\kappa\ge 3$. If every function from $\mcal H$ is Lipschitz continuous,  then 
$\sup_{n=0,\ldots,N}(\mbf E|Y^{SRK}_n-\widetilde{Y}_n|^2)^{1/2}\le Kh^{3/2}. $
\end{lem}
\noindent\emph{Proof} Comparing \eqref{Ybar2} and \eqref{Ytilde} gives a one-step mean error $O(h^{5/2})$ and an $L^2$ error $O(h^2)$. The moment and stability conditions in Theorem \ref{fundamental} follow from the Lipschitz assumptions on the functions in $\mcal H$, so the global comparison is $O(h^{3/2})$.

\begin{rem}
	We remark that  a sufficient condition, to make every function belonging to $\mcal H$	Lipschitz continuous, is that $f$ and $g$ are bounded, $f\in\mbf C_b^3(\mbb R^d)$, and $g\in\mbf C_b^4(\mbb R^d)$. In addition, an example of unbounded $f$ and $g$ satisfying the conditions of Lemma \ref{YSRKn-Ytilden} is $f(x)=g(x)=\ln(1+x^2)$, $x\in\mbb R$.
\end{rem}

By Theorem \ref{SRK1converge} and Lemma \ref{YSRKn-Ytilden}, we have that $\widetilde{Y}_n$ converges to $Y(t_n)$ of root-mean-square order $1$. Thus, $\{\widetilde{Y}_n\}_{n=0}^N$  is an appurtenant method of \eqref{SRK1}. In fact, we can prove that $\sup_{n=0,1,\ldots,N}\|\widetilde{Y}_n-Y(t_n)\|_{\mbf L^p(\Omega)}\le K(p)h$ for any $p\ge 2$.  
\begin{lem}\label{Y-Ytilde}
	Let the assumptions of Lemma \ref{YSRKexpan} hold. If every function from $\mcal H$ is Lipschitz continuous, then for any $p\ge 1$,
$\sup_{n=0,1,\ldots,N}\|\widetilde{Y}_n-Y(t_n)\|_{\mbf L^p(\Omega)}\le K(p)h.$
\end{lem}
\begin{proof}
	Using the one-step It\^o--Taylor expansion stated in Lemma \ref{YSRKexpan} and the identity \eqref{DW2},
	\begin{align*}
		Y_{t,y}(t+h)=y+\Delta Wg(y)+\bar f(y)h+\nabla g(y)g(y)\int_t^{t+h}(W(s)-W(t))\ud W(s)+R^1_{Y}.
	\end{align*}
	Comparing $Y_{t,y}(t+h)$ and $\widetilde Y_{t,y}(t+h)$, one can verify condition (A1) of Theorem \ref{fundamental} with $p_1=2$ and $p_2=\frac{3}{2}$. The moment bounds required in (A2) follow from the estimates for $Y$ and $\widetilde Y$ established above, while (A3) follows from \cite[Lemma 2.2]{ZhangZQ13}. The conclusion then follows from Theorem \ref{fundamental}.
\end{proof} 

For convenience of statement, we introduce the following assumption which implies Assumption \ref{assum1}.
\begin{assum}\label{assum2}
	Assume that $f\in\mbf C^3(\mbb R^d)$, $g\in\mbf C^4(\mbb R^d)$, $\mcal D^3f,\,\mcal D^4g\in\mbf F$ and every function from $\mcal H$ is Lipschitz continuous. Moreover, assume that $\alpha^\top e=\beta^\top e=1$, $\beta^\top(Be)=\frac{1}{2}$ and $\kappa\ge 3$.
\end{assum} 
Under Assumption \ref{assum2}, $\widetilde{Y}_N$ and $Y^{SRK}_N$ have the same asymptotic error distribution.
\begin{lem}\label{samedistri}
	Let Assumption \ref{assum2} hold. Then $N(Y^{SRK}_N-Y(T))$ and $N(\widetilde{Y}_N-Y(T))$ have the same limit distribution as $N\to\infty$, if one of them converges in distribution.
\end{lem}
\begin{proof}
	Note that $N(Y^{SRK}_N-Y(T))=N(\widetilde{Y}_N-Y(T))+Q_N$ with $Q_N=N(Y^{SRK}_N-\widetilde{Y}_N)$. By Lemma \ref{YSRKn-Ytilden}, $Q_N$ converges to $0$ in probability. Thus, the proof is finished as a result of Slutsky's theorem.
\end{proof}

\section{Second application: asymptotic error distribution for stochastic Runge--Kutta method}\label{SecAsymptoticSRK}
\subsection{Limit distribution of normalized error process}
In this section, we give the asymptotic error distribution of 
$Y_N^{SRK}$, which is that of $\widetilde Y_N$ due to Lemma \ref{samedistri}. For this end, we introduce the continuous version of $\{\widetilde Y_n\}_{n=0}^N$:
\begin{align}\label{Ytildet1}
	&\;\widetilde{Y}(t)= Y_0+\int_{0}^t\bar f(\widetilde{Y}(\kappa_N(s)))\ud s+\int_{0}^tg(\widetilde{Y}(\kappa_N(s)))\ud W(s)\nonumber\\
	&\;+\int_{0}^t\int_{\kappa_N(s)}^s\nabla g(\widetilde{Y}(\kappa_N(r)))g(\widetilde{Y}(\kappa_N(r)))\ud W(r)\ud W(s)\notag\\
	&\;+h\int_{0}^t\big[F_1(\widetilde{Y}(\kappa_N(s)))+3F_2(\widetilde{Y}(\kappa_N(s)))\big]\ud W(s) \nonumber\\
	&\;+6\int_{0}^t\int_{\kappa_N(s)}^s\int_{\kappa_N(s)}^rF_2(\widetilde{Y}(\kappa_N(u)))\ud W(u)\ud W(r)\ud W(s)\notag\\
	&\;+h\int_{0}^t\big[\alpha^\top(Ae)\nabla f(\widetilde{Y}(\kappa_N(s)))f(\widetilde{Y}(\kappa_N(s)))+F_3(\widetilde{Y}(\kappa_N(s)))+3F_4(\widetilde{Y}(\kappa_N(s)))\big]\ud s,
\end{align}
where $\kappa_N(s)=\lfloor \frac{s}{h}\rfloor h=\lfloor \frac{Ns}{T}\rfloor \frac{T}{N}$.
Then one has $\widetilde{Y}(t_n)=\widetilde{Y}_n$, $n=0,\ldots,N$ from \eqref{Ytilden} and \eqref{Ytildet1}. It is not difficult to show that $\widetilde{Y}$ converges to $Y$ with strong order $1$ by Lemma \ref{Y-Ytilde}.

\begin{lem}\label{Yt-Ytildet}
	Let Assumption \ref{assum2} hold. Then for any $p\ge 1$, 
$\sup_{t\in[0,T]}\|Y(t)-\widetilde{Y}(t)\|_{\mbf L^p(\Omega)}\le K(p,T)h.$
\end{lem}

We are in the position to give an expansion for the normalized error $N(\widetilde{Y}(t)-Y(t))$.
\begin{lem}\label{UNexpression}
	Denote $U^N(t):=N(\widetilde{Y}(t)-Y(t))$, $t\in[0,T]$. Let Assumption \ref{assum2} hold. Then $U^N$ has the following representation
	\begin{align}
		&\;U^N(t)=\int_0^t\nabla \bar f(Y(s))U^N(s)\ud s+\int_0^t\nabla g(Y(s))U^N(s)\ud W(s)+\sum_{i=1}^{4}I_i^N(t)+R_{U^N}(t).	\label{UNtsimple}
	\end{align}	
	Here,
	\begin{align*}
		&\; I_1^N(t)=T\int_0^t\big[\alpha^\top(Ae)\nabla f(\widetilde{Y}(\kappa_N(s)))f(\widetilde{Y}(\kappa_N(s)))+F_3(\widetilde{Y}(\kappa_N(s)))+3F_4(\widetilde{Y}(\kappa_N(s)))\big]\ud s,\\
		&\;I_2^N(t)=T\int_0^t\big[F_1(\widetilde{Y}(\kappa_N(s)))+3F_2(\widetilde{Y}(\kappa_N(s)))-\nabla\bar f(\widetilde{Y}(\kappa_N(s)))g(\widetilde{Y}(\kappa_N(s)))\big]\ud W(s),\\
		&\;I_3^N(t)=T\int_0^t\big[\nabla \bar f(\widetilde{Y}(\kappa_N(s)))\bar f(\widetilde{Y}(\kappa_N(s)))+\frac{1}{2}\mcal D^2\bar f(\widetilde{Y}(\kappa_N(s)))(g(\widetilde{Y}(\kappa_N(s))),g(\widetilde{Y}(\kappa_N(s))))\big]\\
		&\;\big(\lfloor\frac{Ns}{T}\rfloor-\frac{Ns}{T}\big)\ud s,\\
		&\;I_4^N(t)=T\int_0^tG_1(\widetilde{Y}(\kappa_N(s)))(\frac{Ns}{T}-\lfloor\frac{Ns}{T}\rfloor)\ud W(s)\\
		&\;\qquad\qquad+N\int_0^tG_2(\widetilde{Y}(\kappa_N(s)))\int_{\kappa_N(s)}^s\int_{\kappa_N(s)}^r\ud W(u)\ud W(r)\ud W(s),
	\end{align*}
	where $G_1(y):=\nabla \bar f(y)g(y)-\nabla g(y)\bar f(y)-\frac{1}{2}\mcal D^2g(y)(g(y),g(y))$, $G_2(y)=6F_2(y)-(\nabla g(y))^2g(y)-\mcal D^2g(y)(g(y),g(y))$, $y\in\mbb R^d$ and $R_{U^N}(t)$ is the remainder with $\sup_{t\in[0,T]}\mbf E|R_{U^N}(t)|^2\le Kh$.
\end{lem}

\begin{lem}\label{UN-UtildeN}
	Let Assumption \ref{assum2} hold. Then for any $t\in[0,T]$, $\lim\limits_{N\to\infty}\mbf E|U^N(t)-\widetilde U^N(t)|^2=0$. Here, $\widetilde U^N$ is the strong solution of the following equation 
	\begin{align}\label{UtildeN}
		\widetilde U^N(t)=\int_0^t\nabla \bar f(Y(s))\widetilde{U}^N(s)\ud s+\int_0^t\nabla g(Y(s))\widetilde U^N(s)\ud W(s)+\sum_{i=1}^3I_i(t)+I^N_4(t),~t\in[0,T],
	\end{align}	
	where
	\begin{align*}
		I_1(t)&=T\int_0^t\big[\alpha^\top(Ae)\nabla f(Y(s))f(Y(s))+F_3(Y(s))+3F_4(Y(s))\big]\ud s, \\
		I_2(t)&=T\int_0^t\big[F_1(Y(s))+3F_2(Y(s))-\nabla\bar f(Y(s))g(Y(s))\big]\ud W(s),\\
		I_3(t)&=-\frac{T}{2}\int_0^t\big[\nabla \bar f(Y(s))\bar f(Y(s))+\frac{1}{2}\mcal D^2\bar f(Y(s))(g(Y(s)),g(Y(s)))\big]\ud s.
	\end{align*}
\end{lem}
\noindent For Lemma \ref{UNexpression}, subtract the It\^o equation for $Y$ from \eqref{Ytildet1}, expand the coefficients around $Y(s)$, and use stochastic Fubini to identify $I_1^N,\ldots,I_4^N$; all omitted terms have pointwise $L^2$-order $h^{1/2}$. For Lemma \ref{UN-UtildeN}, the coefficient differences converge in $L^2$ by Lemma \ref{Y-Ytilde}, while the remainder satisfies the same bound. The It\^o isometry and Gronwall's inequality then give the assertion.

\begin{lem}\label{IN4}
	Let Assumption \ref{assum2} hold. Then $I_4^N\overset{stably}{\Longrightarrow}I_4$ as $\mbf C([0,T];\mbb R^d)$-valued random variables as $N\to\infty$. Here, $I_4=\{I_4(t),t\in[0,T]\}$ is the stochastic process defined by
	\begin{align*}
		I_4(t)=\frac{T}{2}\int_0^tG_1(Y(s))\ud W(s)+\frac{T}{\sqrt{12}}\int_0^tG_1(Y(s))\ud \widetilde{W}_1(s)+\frac{T}{\sqrt{6}}\int_0^tG_2(Y(s))\ud \widetilde{W}_2(s),
	\end{align*} 
	where $(\widetilde{W}_1,\widetilde{W}_2)$ is a two-dimensional standard Brownian motion  independent of $W$.
\end{lem}
\begin{proof}
	The proof is mainly based on \cite[Theorem 4-1]{Jacod97}.  Let $I_4^{N,i}$ denote the $i$th component of $I_4^N$ for $i=1,\ldots,d$.  Recall that
	\begin{align*}
		I_4^N(t)=&\;T\int_0^tG_1(\widetilde{Y}(\kappa_N(s)))(\frac{Ns}{T}-\lfloor\frac{Ns}{T}\rfloor)\ud W(s)
		\\&\;+N\int_0^tG_2(\widetilde{Y}(\kappa_N(s)))\int_{\kappa_N(s)}^s\int_{\kappa_N(s)}^r\ud W(u)\ud W(r)\ud W(s).
	\end{align*}
	It is not hard to see $\LL I^{N,i}_4,W\RR_t=B_1^N(t)+B_2^N(t)$ with
	\begin{align*}
		B_1^N(t)=&\;T\int_0^tG_1^i(\widetilde{Y}(\kappa_N(s)))\big(\frac{Ns}{T}-\lfloor\frac{Ns}{T}\rfloor\big)\ud s,\\
		B_2^N(t)=&\;N\int_0^tG_2^i(\widetilde{Y}(\kappa_N(s)))\int_{\kappa_N(s)}^s\int_{\kappa_N(s)}^r\ud W(u)\ud W(r)\ud s.
	\end{align*}
	The coefficient convergence estimate proved in the construction of the auxiliary process, together with \cite[Proposition 4.2]{HJWY24}, gives $\lim\limits_{N\to\infty}\mbf E|B_1^N(t)-\frac{T}{2}\int_0^tG_1^i(Y(s))\ud s|^2=0$. The iterated-integral remainder satisfies $\lim\limits_{N\to\infty}\mbf E|B_2^N(t)|^2=0$ by the same It\^o-isometry estimate. Thus, for any $t\in[0,T]$ and $i=1,\ldots,d$,
	\begin{align}
		\LL I^{N,i}_4,W\RR_t\to \frac{T}{2}\int_0^tG_1^i(Y(s))\ud s~\text{in probability},~\forall~t\in[0,T],~i=1,\ldots,d.\label{IN4converge1}
	\end{align} 
	
	Next, we derive the limit of $\LL I_4^{N,i},I_4^{N,j}\RR_t$. A direct computation leads to
	\begin{align*}
		&\;\LL I_4^{N,i},I_4^{N,j}\RR_t=T^2\int_0^tG_1^i(\widetilde{Y}(\kappa_N(s)))G_1^j(\widetilde{Y}(\kappa_N(s)))	\big(\frac{Ns}{T}-\lfloor\frac{Ns}{T}\rfloor\big)^2\ud s\\
		&\;+N^2\int_{0}^tG_2^i(\widetilde{Y}(\kappa_N(s)))G_2^j(\widetilde{Y}(\kappa_N(s)))\big(\int_{\kappa_N(s)}^s\int_{\kappa_N(s)}^r\ud W(u)\ud W(r)\big)^2\ud s+B_3^N(t),
	\end{align*}
	where $B_3^N(t)=TN\int_0^t\big[G_1^i(\widetilde{Y}(\kappa_N(s)))G_2^j(\widetilde{Y}(\kappa_N(s)))+G_1^j(\widetilde{Y}(\kappa_N(s)))G_2^i(\widetilde{Y}(\kappa_N(s)))\big]\\ \int_{\kappa_N(s)}^s\int_{\kappa_N(s)}^r\ud W(u)\ud W(r)\big(\frac{Ns}{T}-\lfloor\frac{Ns}{T}\rfloor\big)\ud s$.
	It follows from the It\^o formula and \eqref{DW2} that
	\begin{align*}
		&\;\Big(\int_{\kappa_N(s)}^s\int_{\kappa_N(s)}^r\ud W(u)\ud W(r)\Big)^2\\
		=&\;2\int_{\kappa_N(s)}^s\int_{\kappa_N(s)}^r(W(u)-W(\kappa_N(s)))\ud W(u)\big(W(r)-W(\kappa_N(s))\big)\ud W(r)\\
		&\;+2\int_{\kappa_N(s)}^s\int_{\kappa_N(s)}^r(W(u)-W(\kappa_N(s)))\ud W(u)\ud r+\frac{1}{2}(s-\kappa_N(s))^2.
	\end{align*}
	Combining the above  relations, we have
	\begin{align*}
		\LL I_4^{N,i},I_4^{N,j}\RR_t=&\;T^2\int_0^tG_1^i(\widetilde{Y}(\kappa_N(s)))G_1^j(\widetilde{Y}(\kappa_N(s)))	\big(\frac{Ns}{T}-\lfloor\frac{Ns}{T}\rfloor\big)^2\ud s\\
		&\;+\frac{T^2}{2}\int_0^tG_2^i(\widetilde{Y}(\kappa_N(s)))G_2^j(\widetilde{Y}(\kappa_N(s)))	\big(\frac{Ns}{T}-\lfloor\frac{Ns}{T}\rfloor\big)^2\ud s+B_3^N(t)+B_4^N(t),	
	\end{align*}
	where
	\begin{align*}
		&\;B_4^N(t)=2N^2\int_0^tG_2^i(\widetilde{Y}(\kappa_N(s)))G_2^j(\widetilde{Y}(\kappa_N(s)))\Big\{\int_{\kappa_N(s)}^s\Big[\int_{\kappa_N(s)}^r(W(u)-W(\kappa_N(s)))\ud W(u)\\
		&\;\big(W(r)-W(\kappa_N(s))\big)\Big]\ud W(r)+\int_{\kappa_N(s)}^s\int_{\kappa_N(s)}^r(W(u)-W(\kappa_N(s)))\ud W(u)\ud r\Big\}\ud s.
	\end{align*}
	The argument of conditional expectations gives that $\mbf E|B_3^N(t)|^2+\mbf E|B_4^N(t)|^2\le Kh\to 0$, as $N\to\infty$. Using \cite[Proposition 4.2]{HJWY24} yields that $\int_0^tG_l^i(\widetilde{Y}(\kappa_N(s)))G_l^j(\widetilde{Y}(\kappa_N(s)))	\big(\frac{Ns}{T}-\lfloor\frac{Ns}{T}\rfloor\big)^2\ud s$ converges to $\frac{1}{3}\int_0^tG_l^i(Y(s))G_l^j(Y(s))	\ud s$ in $\mbf L^2(\Omega,\mcal F,\mbf P;\mbb R)$ for $l=1,2$. In this way, we obtain
	\begin{align}
		\LL I_4^{N,i},I_4^{N,j}\RR_t\to \frac{T^2}{3}\int_0^tG_1^i(Y(s))G_1^j(Y(s))	\ud s+\frac{T^2}{6}\int_0^tG_2^i(Y(s))G_2^j(Y(s))\ud s \label{IN4converge2}
	\end{align}
	in probability for any $t\in[0,T]$ and $i,j=1,\ldots,d$.
	
	Combining \eqref{IN4converge1}-\eqref{IN4converge2} and using \cite[Theorem 4-1]{Jacod97}, we obtain $I_4^N\overset{stably}{\Longrightarrow} \Psi$ as $\mbf C([0,T];\mbb R^d)$-valued random variables as $N\to\infty$, and 
	\begin{align}
		&\;\LL \Psi^i,W\RR_t=\frac{T}{2}\int_0^tG_1^i(Y(s))\ud s, \label{VWcross}\\
		&\;\LL \Psi^i,\Psi^j\RR_t=\frac{T^2}{3}\int_0^tG_1^i(Y(s))G_1^j(Y(s))	\ud s+\frac{T^2}{6}\int_0^tG_2^i(Y(s))G_2^j(Y(s))\ud s.\label{VVcross}
	\end{align}
	Further, it follows from \cite[Proposition 1-4]{Jacod97} that $\Psi^i$, $i=1,\ldots,d$ can be represented as
	\begin{align*}
		\Psi^i(t)=\int_{0}^tu^i(s)\ud W(s)+\sum_{k=1}^{d'}\int_0^tv^{i,k}(s)\ud \widetilde{W}_k(s),
	\end{align*}
	where $\widetilde{W}=\big(\widetilde{W}_1,\ldots,\widetilde{W}_{d'}\big)$ is some $d'$-dimensional standard Brownian motion and independent of $W$.  By \eqref{VWcross}, we have $u^i(s)=\frac{T}{2}G_1^i(Y(s))$, $i=1,\ldots,d$, which along with \eqref{VVcross} gives
	\begin{align*}
		\sum_{k=1}^{d'}v^{i,k}(s)v^{j,k}(s)=\frac{T^2}{12}G_1^i(Y(s))G_1^j(Y(s))+\frac{T^2}{6}G_2^i(Y(s))G_2^j(Y(s)).
	\end{align*}
	We can take $d'=2$,  $v^{i,1}(s)=\frac{T}{\sqrt{12}}G_1^i(Y(s))$, and $v^{i,2}(s)=\frac{T}{\sqrt{6}}G_2^i(Y(s))$, $i=1,\ldots,d$. Thus the proof is complete by letting $I_4=\Psi$. 
\end{proof}

With previous preparation, now we can give the limit distribution of $U^N(t)$. 
\begin{theo}\label{maintheorem1}
	Let Assumption \ref{assum2} hold. Then for any $t\in[0,T]$, we have $U^N(t)\overset{d}{\Longrightarrow}U(t)$ as $N\to\infty$, and thus $N(Y^{SRK}_N-Y(T))\overset{d}{\Longrightarrow}U(T)$. Here, $U=\{U(t)\}_{t\in[0,T]}$ is the strong solution of the following equation
	\begin{align}
		&\;U(t)=\int_0^t\nabla \bar{f}(Y(s))U(s)\ud s+\int_0^t\nabla g(Y(s))U(s)\ud W(s)\nonumber\\
		&\;	+T\int_0^t\big[\alpha^\top(Ae)\nabla f(Y(s))f(Y(s))+F_3(Y(s))+3F_4(Y(s))\big]\ud s\nonumber\\
		&\;-\frac{T}{2}\int_0^t\big[\nabla \bar f(Y(s))\bar{f}(Y(s))+\frac{1}{2}\mcal D^2\bar f(Y(s))(g(Y(s)),g(Y(s)))\big]\ud s\nonumber\\
		&\;+T\int_0^t\big[F_1(Y(s))+3F_2(Y(s))-\nabla\bar f(Y(s))g(Y(s))\big]\ud W(s)\nonumber\\
		&\;+T\int_0^t\big[\nabla\bar f(Y(s))g(Y(s))-\nabla g(Y(s))\bar f(Y(s))-\frac{1}{2}\mcal D^2 g(Y(s))(g(Y(s)),g(Y(s)))\big]\ud \Big(\frac{W(s)}{2}+\frac{\widetilde{W}_1(s)}{\sqrt{12}}\Big)\nonumber\\
		&\;+\frac{T}{\sqrt{6}}\int_0^t\big[6F_2(Y(s))-(\nabla g(Y(s)))^2g(Y(s))-\mcal D^2 g(Y(s))(g(Y(s)),g(Y(s)))\big]\ud \widetilde W_2(s). \label{Ut}
	\end{align}
	Here,  $(\widetilde{W}_1,\widetilde{W}_2)$ is a two-dimensional standard Brownian motion and independent of $W$. In addition, $F_i$, $i=1,2,3,4$ are those defined in Lemma \ref{YSRKexpan}.
\end{theo}
\begin{proof}
	By Proposition \ref{Pro1} and Lemma \ref{IN4}, $(I_1,I_2,I_3,I_4^N)\overset{stably}{\Longrightarrow}(I_1,I_2,I_3,I_4)$ in $\mbf C([0,T];\mbb R^d)^{\otimes4}$ as $N\to\infty$. 	Note that a continuous mapping maps a family of random variables converging stably in law into another family of random variables converging stably in law, which can be directly shown according to the definition of stable convergence. Thus, we obtain $\Phi^N:=I_1+I_2+I_3+I_4^N\overset{stably}{\Longrightarrow}\Phi:=I_1+I_2+I_3+I_4$ in $\mbf C([0,T];\mbb R^d)$. 
	
	Note that the coefficients $\nabla \bar f$ and $\nabla g$   of \eqref{UtildeN} are bounded and $\mcal D^2\bar f,\mcal D^2g\in\mbf F$. In addition,
	a direct computation leads to 
	\begin{gather*}
		\sup_{N\ge 1}\sup_{t\in[0,T]}\mbf E|\Phi^N(t)|^4+\sup_{t\in[0,T]}\mbf E|\Phi(t)|^4<\infty,\\
		\sup_{N\ge 1}\mbf E|\Phi^N(t)-\Phi^N(s)|^2+\mbf E|\Phi(t)-\Phi(s)|^2\le K|t-s|.
	\end{gather*}
	The above facts and \eqref{Ybound}-\eqref{Yt-Ys} enable us to
	apply Theorem \ref{stablecon} to inferring $\widetilde{U}^N\overset{d}{\Longrightarrow} V$, with $V$ solving
	\begin{align*}
		V(t)=\int_0^t\nabla\bar f(Y(s))V(s)\ud s+\int_0^t\nabla g(Y(s))V(s)\ud W(s)+\sum_{i=1}^4I_i(t),~t\in[0,T].
	\end{align*}
	Note that $V$ is nothing but $U$. Consequently, for any $t\in[0,T]$, $\widetilde{U}^N(t)\overset{d}{\Longrightarrow} U(t)$,
	which together with Lemma \ref{UN-UtildeN} and Slutsky's theorem yields ${U}^N(t)\overset{d}{\Longrightarrow} U(t)$ for any $t\in[0,T]$. Noting $U^N(T)=N(\widetilde{Y}_N-Y(T))$ and using Lemma \ref{samedistri}, we obtain $N(Y^{SRK}_N-Y(T))\overset{d}{\Longrightarrow} U(T)$. Thus, the proof is complete.
\end{proof}

\subsection{Properties of limit equation}
In this part, we study the properties of $U(t)$ and reveal the relationship between  $\mbf E|U(T)|^2$ and the weak convergence of the SRK method \eqref{SRK1}. 
A direct computation yields
\begin{align*}
	(\nabla \bar{f})\bar f&=(\nabla f)f+\frac{1}{2}(\nabla f\nabla g)g+\frac{1}{2}\mcal D^2g(f,g)+\frac{1}{4}\mcal D^2g(g,(\nabla g)g)+\frac{1}{2}(\nabla g)^2f+\frac{1}{4}(\nabla g)^3g, \\
	\mcal D^2\bar f(g,g)&=\mcal D^2f(g,g)+\frac{1}{2}(\nabla g)\mcal D^2g(g,g)+\mcal D^2g(g,(\nabla g)g)+\frac{1}{2}\mcal D^3g(g,g,g),\\
	(\nabla \bar f)g&=(\nabla f)g+\frac{1}{2}\mcal D^2g(g,g)+\frac{1}{2}(\nabla g)^2g.
\end{align*}
Plugging the above three relations into \eqref{Ut} and using the definitions of $F_i$, $i=1,2,3,4$ (see Lemma \ref{YSRKexpan}), we have
\begin{align}
	U(t)=	&\;\int_0^t\nabla \bar f(Y(s))U(s)\ud s+\int_0^t\nabla g(Y(s))U(s)\ud W(s)+T\int_0^t\mbf H_1(Y(s))\ud s \notag\\
	&\;+T\int_0^t\mbf H_2(Y(s))\ud W(s)
	+\frac{T}{\sqrt{12}}\int_0^t\big[\nabla f(Y(s))g(Y(s))-\nabla g(Y(s))f(Y(s))\big]\ud \widetilde{W}_1(s)\nonumber\\
	&\;+\frac{T}{\sqrt{6}}\int_0^t\mbf H_3(Y(s))\ud \widetilde W_2(s),\label{Utsimple}
\end{align}
where 
\begin{align*}
	&\;\mbf H_1(y)=\big(\alpha^\top(Ae)-\frac{1}{2}\big)\nabla f(y)f(y)+\big(\alpha^\top(B(Be))-\frac{1}{4}\big)\nabla f(y)\nabla g(y)g(y) \notag\\
	&\;+\frac{1}{2}\big(\alpha^\top(Be)^2-\frac{1}{2}\big)\mcal D^2f(y)(g(y),g(y))+\beta^\top (A(Be))\nabla g(y)\nabla f(y)g(y) \notag\\
	&\;+\frac{3}{2}\big(\beta^\top(B(Be)^2)-\frac{1}{12}\big)\nabla g(y)\mcal D^2g(y)(g(y),g(y))+\big(\beta^\top(B(Ae))-\frac{1}{4}\big)(\nabla g(y))^2f(y)\notag \\
	&\;+3\big(\beta^\top(B(B(Be)))-\frac{1}{24}\big)(\nabla g(y))^3g(y)+\big(\beta^\top((Ae)*(Be))-\frac{1}{4}\big)\mcal D^2g(y)(f(y),g(y))\notag\\
	&\;+3\big(\beta^\top((Be)*(B(Be)))-\frac{1}{8}\big)\mcal D^2g(y)(g(y), \nabla g(y)g(y)) +\frac{1}{2}\big(\beta^\top(Be)^3-\frac{1}{4}\big)\mcal D^3g(y)(g(y),g(y),g(y)),\\
	&\; \mbf H_2(y)=\big(\alpha^\top(Be)-\frac{1}{2}\big)\nabla f(y)g(y)+\big(\beta^\top(Ae)-\frac{1}{2}\big)\nabla g(y)f(y)\notag\\
	&\;+3\big(\beta^\top(B(Be))-\frac{1}{6}\big)(\nabla g(y))^2g(y)+\frac{3}{2}\big(\beta^\top(Be)^2-\frac{1}{3}\big)\mcal D^2g(y)(g(y),g(y)),\notag\\
	&\;\mbf H_3(y)=6\big(\beta^\top(B(Be))-\frac{1}{6}\big)(\nabla g(y))^2g(y)+3\big(\beta^\top(Be)^2-\frac{1}{3}\big)\mcal D^2 g(y)(g(y),g(y)).
\end{align*}

Define
\begin{align}
	\eta_1=&\;(\alpha^\top(Ae)-\frac{1}{2})^2+(\alpha^\top(B(Be))-\frac{1}{4})^2+(\alpha^\top(Be)^2-\frac{1}{2})^2+(\beta^\top(A(Be)))^2+(\beta^\top(B(Be)^2)-\frac{1}{12})^2\notag\\
	&\;+(\beta^\top(B(Ae))-\frac{1}{4})^2+(\beta^\top(B(B(Be)))-\frac{1}{24})^2+(\beta^\top((Ae)*(Be))-\frac{1}{4})^2\notag\\
	&\;+(\beta^\top((Be)*(B(Be)))-\frac{1}{8})^2+(\beta^\top(Be)^3-\frac{1}{4})^2+(\alpha^\top(Be)-\frac{1}{2})^2+(\beta^\top(Ae)-\frac{1}{2})^2\notag\\
	&\;+(\beta^\top(B(Be))-\frac{1}{6})^2+(\beta^\top(Be)^2-\frac{1}{3})^2.\label{eta1}
\end{align} 
Note that $\int_0^t\nabla \bar f(Y(s))U(s)\ud s+\int_0^t\nabla g(Y(s))U(s)\ud W(s)=\int_0^t\nabla  f(Y(s))U(s)\ud s+\int_0^t\nabla g(Y(s))U(s)\circ\ud W(s)$. 
Next, we have the following upper estimate for $\mbf E|U(T)|^2$.
\begin{theo}\label{upperbound}
	Let Assumption \ref{assum2} hold and $T\ge 1$.  Then the following hold.
	\begin{itemize}
	\item [(1)]	$\lim\limits_{h\to 0}\frac{(\mbf E|Y^{SRK}_N-Y(T)|^2)^{1/2}}{h}=\frac{(\mbf E|U(T)|^2)^{1/2}}{T}$. 
	
	\item [(2)] There exists $L_1>0$ independent of \(T\) and of the coefficients of the SRK method such that
$	\mbf E|U(T)|^2\le e^{L_1T}(1+\eta_1)T^3.$
	\end{itemize} 
\end{theo}
\begin{proof}
(1) It follows from Theorem \ref{maintheorem1}	and the continuous mapping theorem that $|N(Y^{SRK}_N-Y(T))|^2\overset{d}{\Longrightarrow} |U(T)|^2$ as $N\to\infty$. By Theorem \ref{SRK1converge}, we have $\sup_{N\ge 1}\mbf E|N(Y^{SRK}_N-Y(T))|^3<+\infty$ due to $\kappa\ge 3$, which implies that $\{|N(Y^{SRK}_N-Y(T))|^2\}_{N\ge 1}$ is uniformly  integrable. Then applying \cite[Theorem 3.5]{Billing1992} yields  $\lim\limits_{N\to\infty}N^2\mbf E|Y^{SRK}_N-Y(T)|^2=\mbf E|U(T)|^2$. Thus, the first conclusion comes from $Nh=T$.

(2)	In this proof, we denote by $L$ a generic constant independent of $T$, which may vary for each appearance.	Note that under Assumption \ref{assum2}, every function from $\mcal H$ grows at most linearly, and $\nabla \bar{f}$ and $\nabla g$ are bounded. The It\^o  formula  yields
	\begin{align*}
		&\;\mbf E|U(t)|^2=2\mbf E\int_0^t\LL U(s),\nabla \bar f(Y(s))U(s)\RR\ud s+2T\mbf E\int_0^t\LL U(s),\mbf H_1(Y(s))\ud s\\
		&\;+\mbf E\int_0^t\big|\nabla g(Y(s))U(s)+T\mbf H_2(Y(s))|^2\ud s+\frac{T^2}{12}\mbf E\int_0^t|\nabla f(Y(s))g(Y(s))-\nabla g(Y(s))f(Y(s))|^2\ud s\\
		&\;+\frac{T^2}{6}\mbf E\int_0^t|\mbf H_3(Y(s))|^2\ud s\\
		&\le L\int_0^t\mbf E|U(s)|^2\ud s+LT^2\sum_{i=1}^3\int_0^t\mbf E|\mbf H_i(Y(s))|^2\ud s+LT^2\int_0^t(1+\mbf E|Y(s)|^2)\ud s.
	\end{align*}
	By the definitions of $\eta_1$ (see \eqref{eta1}) and $\mbf H_i$, $i=1,2,3$, we have
	\begin{align}
		\mbf E|U(t)|^2\le&\; L\int_0^t\mbf E|U(s)|^2\ud s+LT^2\int_0^t(1+\eta_1)(1+\mbf E|Y(s)|^2)\ud s. \label{sec4eq6}
	\end{align} 	
	Applying \eqref{Ybound}, one has $\sup_{s\in[0,T]}\mbf E|Y(s)|^2\le Le^{LT}$. This, combined with \eqref{sec4eq6} and the Gronwall inequality, yields the desired result.
\end{proof}

The following lemma shows that the SRK method \eqref{SRK1} is of weak order $2$ provided $\eta_1=0$. Its proof is a standard one-step weak Taylor comparison: when $\eta_1=0$, the moments of the SRK increment agree with those of the exact increment through order $h^2$, and the truncation estimates preserve the $O(h^3)$ local weak defect.
\begin{lem}\label{weakorder}
	Let Assumptions \ref{assum2} hold. Assume that $f\in\mbf C^6(\mbb R^d)$ and $g\in\mbf C^7(\mbb R^d)$ with $\mcal D^6f,\,\mcal D^7g\in\mbf F$. If $\eta_1=0$, then the SRK method \eqref{SRK1} is of weak order $2$, i.e., for any $\varphi\in\mbf C^6(\mbb R^d)$ with $\mcal D^6\varphi\in\mbf F$, 
$	\sup_{n=0,1,\ldots,N}\big|\mbf E\varphi(Y^{SRK}_n)-\mbf E\varphi(Y(t_n))\big|\le Kh^2.$
\end{lem}

Combining Theorems \ref{maintheorem1}, Lemma \ref{weakorder}, and \eqref{Utsimple}, we obtain the following result.
\begin{theo}\label{maintheorem2}
	Let the assumptions of Lemma \ref{weakorder} hold. If $\eta_1=0$, then the SRK method \eqref{SRK1} is of weak order $2$, and the limit distribution of  $N(Y^{SRK}_N-Y(T))$ as $N\to\infty$ takes the following unified form
	\begin{align*}
		U(T)=&\;\int_0^T\nabla \bar f(Y(s))U(s)\ud s
		+\int_0^T\nabla g(Y(s))U(s)\ud W(s)\\
		&\;+\frac{T}{\sqrt{12}}\int_0^T\big[\nabla f(Y(s))g(Y(s))-\nabla g(Y(s))f(Y(s))\big]\ud \widetilde{W}_1(s).
	\end{align*} 
\end{theo}
\begin{rem}\label{inf1}
Theorem \ref{upperbound}(1) indicates that $\frac{(\mbf E|U(T)|^2)^{1/2}}{T}$ is the asymptotic root-mean-square constant of the SRK method \eqref{SRK1}. Theorem \ref{upperbound}(2) further shows that $\eta_1$ governs the coefficient-dependent part of the theoretical upper bound for this constant. We therefore obtain the following conclusions: (1) within the class of strong-order-one SRK methods \eqref{SRK1}, this upper bound is nondecreasing in $\eta_1$; and (2) methods with $\eta_1=0$, which are precisely the weak-order-two methods by Lemma \ref{weakorder}, minimize this upper bound.      
\end{rem}

\section{Numerical experiment}\label{SecNumerical}
We first illustrate the limit theorem for the stochastic midpoint method by a two-dimensional equation driven by two non-commuting diffusion fields. The numerical comparison is made through empirical characteristic functions. We then consider a three-stage parametric SRK family and compare the long-time root-mean-square errors of five members with different values of $\eta_1$.

\subsection{Characteristic-function test for the stochastic midpoint method}\label{SecNumericalMidpoint}
Let $T=1$, $q(x)=2+\dfrac12(\sin x_1+\sin x_2)$, and consider the Stratonovich SDE
\begin{align}\label{numSDE}
	\begin{cases}
		\ud X_1(t)=X_2(t)\ud t+q(X(t))\circ\ud W_1(t),\\
		\ud X_2(t)=X_1(t)\ud t+q(X(t))\circ\ud W_2(t),
	\end{cases}
\end{align}
with $X(0)=(0.8,-0.4)^\top$. Here,  $W_1$ and $W_2$ are independent. In the notation of Section \ref{SecMidpoint},
$f(x)=(x_2,x_1)^\top,~
	g_1(x)=q(x)e_1,~ g_2(x)=q(x)e_2.$

For $N=32,64,128,256,512,1024$, with $h=T/N$, we apply the fully implicit stochastic midpoint rule
\begin{align*}
	X_{n+1}^{\rm mid}=X_n^{\rm mid}
	+h f\left(\frac{X_n^{\rm mid}+X_{n+1}^{\rm mid}}2\right)
	+q\left(\frac{X_n^{\rm mid}+X_{n+1}^{\rm mid}}2\right)
	\left(e_1\Delta\widehat W_n^1+e_2\Delta\widehat W_n^2\right).
\end{align*}
The Brownian increments are truncated componentwise according to \eqref{zeta}, with $\kappa=3$. At every time step, $X_{n+1}^{\rm mid}$ is solved by Newton's method using its analytic $2\times2$ Jacobian. The explicit Euler value
$X_n^{\rm mid}+h f(X_n^{\rm mid})
	+q(X_n^{\rm mid})
	\left(e_1\Delta\widehat W_n^1+e_2\Delta\widehat W_n^2\right)$
is taken as the initial iterate. The iteration is stopped when the maximum norm of the residual is below $10^{-12}$, with at most $12$ iterations allowed. No failure of the nonlinear solver occurred in the simulation reported below.

To approximate both the reference solution and the limiting random variable, we regard $(X,V)$ as a four-dimensional It\^o SDE. Set
$\bar f=f+\frac12\sum_{r=1}^2\nabla g_r\,g_r.$
According to Theorem \ref{midpointlimit}, the system to be discretized is
\begin{align}\label{numXV}
	\ud X(t)=&\;\bar f(X(t))\ud t
	+g_1(X(t))\ud W_1(t)+g_2(X(t))\ud W_2(t),\notag\\
	\ud V(t)=&\;\nabla\bar f(X(t))V(t)\ud t
	+\nabla g_1(X(t))V(t)\ud W_1(t)
	+\nabla g_2(X(t))V(t)\ud W_2(t)\notag\\
	&+\frac{\sqrt T}{2}[g_1,g_2](X(t))\ud\widetilde W(t),
	\qquad X(0)=(0.8,-0.4)^\top,\quad V(0)=0,
\end{align}
where $\widetilde W$ is a standard Brownian motion independent of $(W_1,W_2)$. The coefficients and their derivatives in \eqref{numXV} are evaluated analytically. We apply the Euler--Maruyama method with untruncated increments to \eqref{numXV} on a uniform grid with $N_{\rm ref}=2^{18}.$
The terminal value of the $X$-component is used as $X_{\rm ref}(T)$, while the $V$-component gives the numerical approximation of $V(T)$. 

The experiment uses $M=10^5$ sample paths. For each $N$, we form
$	E_N^{\rm mid}=\sqrt N\big(X_N^{\rm mid}-X_{\rm ref}(T)\big)$
and compute the empirical characteristic functions
\begin{align*}
	\widehat\chi_N(u)
	=\frac1M\sum_{j=1}^M\exp\big(\mathrm i u^\top E_{N,j}^{\rm mid}\big),
	\qquad
	\widehat\chi_V(u)
	=\frac1M\sum_{j=1}^M\exp\big(\mathrm i u^\top V_j(T)\big).
\end{align*}
We use the same $K=500$ frequencies for every $N$, taking the equally spaced points on the unit circle,
$	u_\ell=	\Big(\cos\big(2\pi(\ell-1)/K\big)
		\sin\big(2\pi(\ell-1)/K\big)\Big)^\top,~ \ell=1,\ldots,K.$
This fixed choice samples the directions uniformly and is independent of the simulated samples. The discrepancy between the two empirical distributions is measured by
\begin{align}\label{numDN}
	\widehat D_N
	=\left[
	\frac1K\sum_{\ell=1}^K
	\left|\widehat\chi_N(u_\ell)-\widehat\chi_V(u_\ell)\right|^2
	\right]^{1/2}.
\end{align}

Figure \ref{fig:midpointECF} shows a monotone decrease of $\widehat D_N$ from
$5.5944\times10^{-2}$ for $N=32$ to $8.9510\times10^{-3}$ for $N=1024$.
The least-squares slope on the log--log scale is $-0.53$. This computation  provides numerical evidence that the distribution of the normalized midpoint error approaches the distribution of the limiting variable in Theorem \ref{midpointlimit}.

\begin{figure}[htbp]
	\centering
\includegraphics[width=0.76\textwidth]{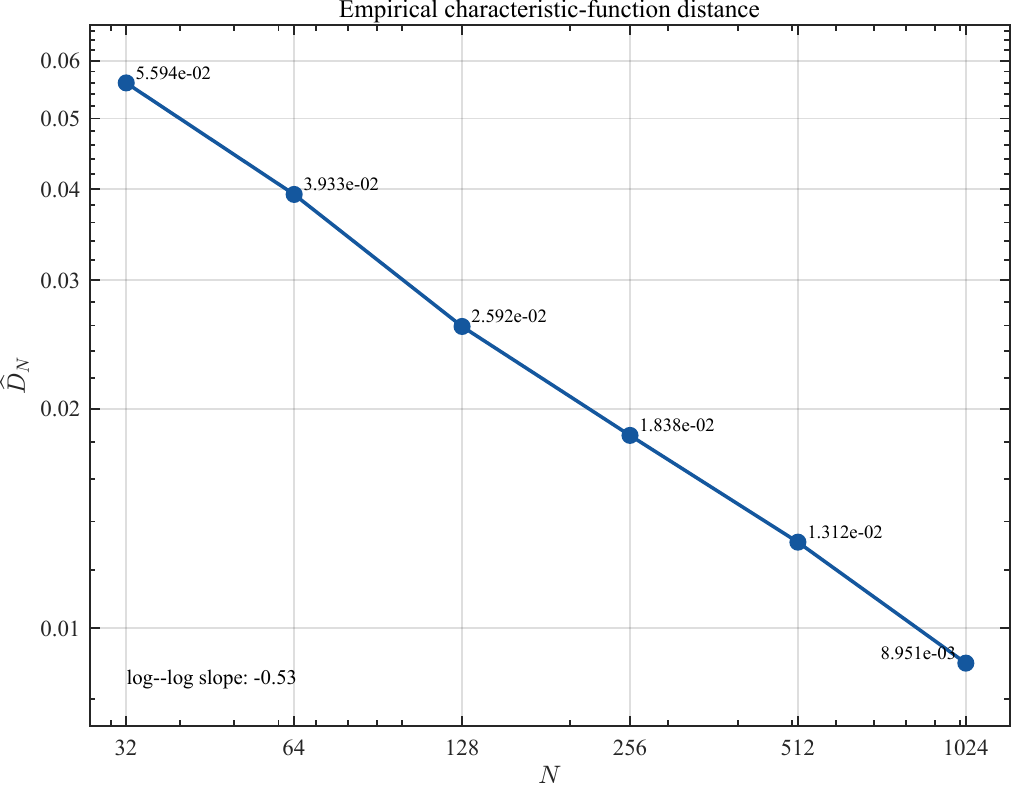}
	\caption{Empirical characteristic-function discrepancy \eqref{numDN} between $E_N^{\rm mid}$ and the fine-grid approximation of $V(T)$. The same $500$ equally spaced points on the unit circle and $10^5$ sample paths are used for all values of $N$.}
	\label{fig:midpointECF}
\end{figure}

\subsection{Long-time errors of a parametric SRK family}\label{SecNumericalSRKFamily}
In this subsection, we first consider the two-dimensional Stratonovich equation driven by a scalar Brownian motion,
\begin{align}\label{numSRKSDE}
	\begin{cases}
		\displaystyle \ud Y_1(t)=\frac{1}{1+Y_2(t)^2}\ud t
		+\ln\big(1+Y_1(t)^2\big)\circ\ud W(t),\\[2mm]
		\displaystyle \ud Y_2(t)=\frac{1}{1+Y_1(t)^2}\ud t
		+\ln\big(1+Y_2(t)^2\big)\circ\ud W(t),
	\end{cases}
	\qquad Y(0)=Y_0\in\mbb R^2.
\end{align}
In other words, with $a(x)=(1+x^2)^{-1}$ and $q(x)=\ln(1+x^2)$,
	$f(y)=\big(a(y_2),a(y_1)\big)^\top$, $g(y)=\big(q(y_1),q(y_2)\big)^\top$.
It follows directly that every function in $\mcal H$ is globally Lipschitz and that the coefficient conditions in Assumption \ref{assum2} are fulfilled. 

To compare the long-time root-mean-square errors of strong-order-one SRK methods with different weak orders, it is useful to vary as few order conditions as possible. We consider \eqref{SRK1} with $s_0=3$ and
\begin{align}\label{parametricSRK}
	A(\theta)=
	\begin{pmatrix}
		\frac12-\theta~~0~~\theta\\
		\frac12-\theta~~0~~\theta\\
		\frac12-\theta~~0~~\theta
	\end{pmatrix},~
	B=
	\begin{pmatrix}
		0~~0~~0\\
		\frac14~~\frac14~~0\\
		0~~1~~0
	\end{pmatrix},~
	\alpha^\top=\begin{pmatrix}\frac12~~0~~\frac12\end{pmatrix},~
	\beta^\top=\begin{pmatrix}\frac16~~\frac23~~\frac16\end{pmatrix}.
\end{align}
We denote the resulting method by $\operatorname{SRK3}(\theta)$. The parameter may take any real value. In the numerical comparison below we take $0\le\theta\le1$. 

A direct calculation gives
	$Ae=\frac12e$,
	$Be=\big(0~~ \frac12~~ 1\big)^\top$,
	$B(Be)=\big(0~~ \frac18~~\frac12\big)^\top$, $A(Be)=\theta e$.
It follows at once that, for every $\theta$,
$\alpha^\top e=\beta^\top e=1,~
	\beta^\top(Be)=\frac12.$
Thus Theorem \ref{SRK1converge} applies to the whole family, and every $\operatorname{SRK3}(\theta)$ has strong order one. Substitution of these relations into \eqref{eta1} shows that all the residuals entering $\eta_1$ vanish except
$\beta^\top\big(A(Be)\big)=\theta.$
Consequently, $\eta_1(\theta)=\theta^2.$
	
The method $\operatorname{SRK3}(0)$ therefore satisfies the weak-order-two conditions in Lemma \ref{weakorder}. If $\theta\ne0$, the condition $\beta^\top(A(Be))=0$ is violated, while the strong-order-one conditions remain unchanged; $\operatorname{SRK3}(\theta)$ then is a weak-order-one method. 

\medskip
\noindent\emph{Long-time root-mean-square errors.}
We next examine the indication in Remark \ref{inf1} by comparing five members of the same family,
\begin{align*}
	\operatorname{SRK3}(0),\quad
	\operatorname{SRK3}(1/4),\quad
	\operatorname{SRK3}(1/2),\quad
	\operatorname{SRK3}(3/4),\quad
	\operatorname{SRK3}(1).
\end{align*}
Their values of $\eta_1$ are, in the same order,
$	0,~ \frac1{16},~ \frac14,~
	\frac9{16},~ 1.$
Thus the comparison changes only the parameter $\theta$ in \eqref{parametricSRK}, while the remaining method coefficients are fixed.

We take $Y_0=(0.8,-0.4)^\top$, $T=50$, and use $M=5000$ independent sample paths. Each SRK approximation is computed with the common step size $h=0.01$. The increments are truncated according to \eqref{zeta}, with $\kappa=3$. The coupled stage equations are solved simultaneously by Newton's method, starting from $Z_{n,i}=Y_n^{SRK}$, $i=1,2,3$. The iteration is stopped when the maximum norm of the residual is at most $10^{-12}$; no more than $12$ iterations are allowed. In this experiment, at most three iterations were required and no nonlinear solve failed.

On the same Wiener paths, a reference solution is computed with step size $h_{\rm ref}=10^{-4}$ by the scalar-noise Stratonovich Milstein method
\begin{align}\label{numSRKreference}
	Y_{{\rm ref},k+1}={}Y_{{\rm ref},k}
	+h_{\rm ref}f(Y_{{\rm ref},k})
	+g(Y_{{\rm ref},k})\delta W_k+\frac12\nabla g(Y_{{\rm ref},k})g(Y_{{\rm ref},k})(\delta W_k)^2.
\end{align}
Here the fine-grid increments $\delta W_k$ are not truncated. Every coarse Brownian increment is formed by summing the corresponding $h/h_{\rm ref}=100$ fine-grid increments before applying the truncation in \eqref{zeta}. The SRK approximations and the reference solution are therefore generated from the same underlying Wiener path. 

At the coarse-grid points $t_n=nh$, we calculate the empirical root-mean-square error
\begin{align}\label{numSRKRMSE}
	\mathcal E_{\rm rms}^{\theta}(t_n)
	=\left(\frac1M\sum_{j=1}^M
	\left|Y_{n,\theta}^{(j)}-Y_{\rm ref}^{(j)}(t_n)\right|^2\right)^{1/2}.
\end{align}
The upper panel of Figure \ref{fig:srkLongTimeRMSE} displays \eqref{numSRKRMSE} over the whole interval. At $T=50$, the five errors, listed in increasing order of $\theta$, are
$	6.0362\times10^{-2},~	6.2690\times10^{-2},~
	7.1192\times10^{-2},~	8.4024\times10^{-2},~	9.9534\times10^{-2}.$
The terminal root-mean-square error increases monotonically with $\eta_1$, and the separation of the curves becomes more pronounced with time. Relative to the weak-order-two method $\operatorname{SRK3}(0)$, the terminal errors of $\operatorname{SRK3}(1/4)$, $\operatorname{SRK3}(1/2)$, $\operatorname{SRK3}(3/4)$, and $\operatorname{SRK3}(1)$ are larger by approximately $3.9\%$, $17.9\%$, $39.2\%$, and $64.9\%$, respectively.

To check the same comparison for a different coefficient pair, we also consider the scalar Stratonovich equation
\begin{align}\label{numScalarSRKSDE}
	\ud Y(t)=\ln\bigl(1+Y(t)^2\bigr)\ud t
	+\frac14Y(t)\circ\ud W(t),
	\qquad Y(0)=\frac12.
\end{align}
This equation also satisfies Assumption \ref{assum2}. 

\begin{figure}[H]
	\centering
\includegraphics[width=0.49\textwidth]{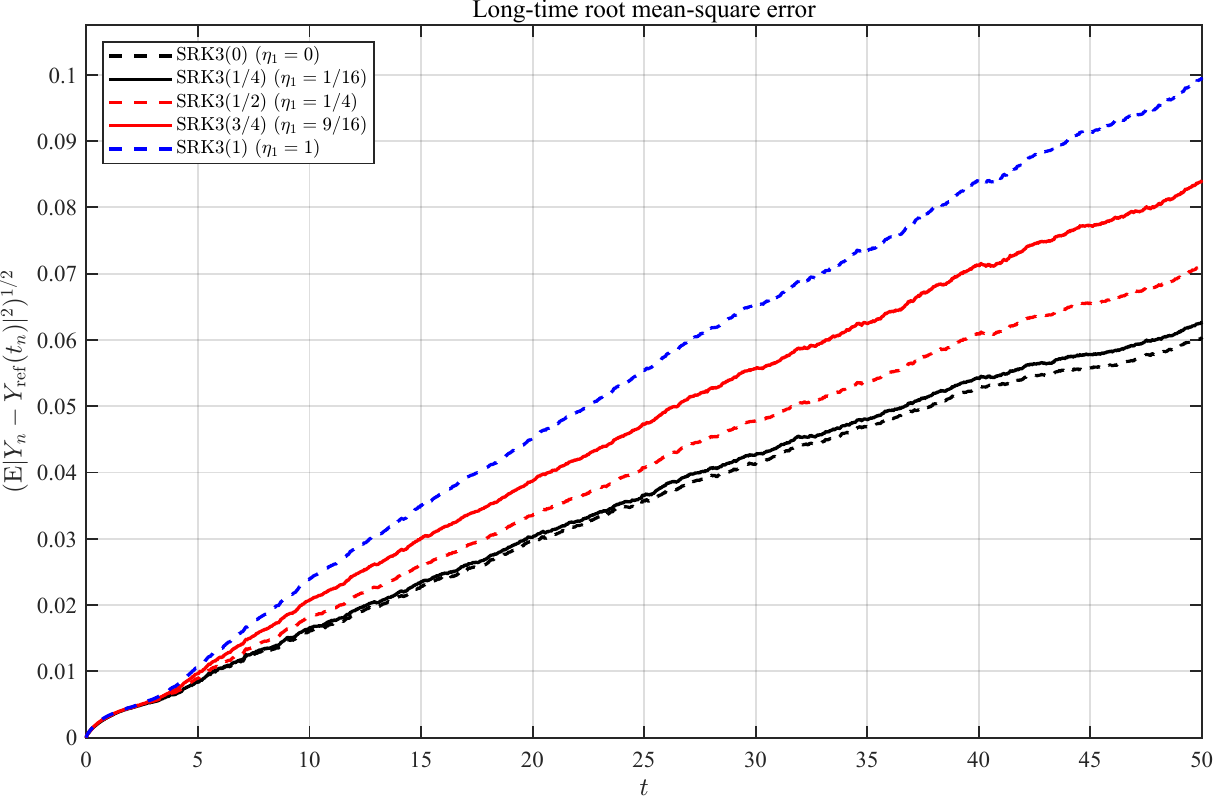}
\includegraphics[width=0.49\textwidth]{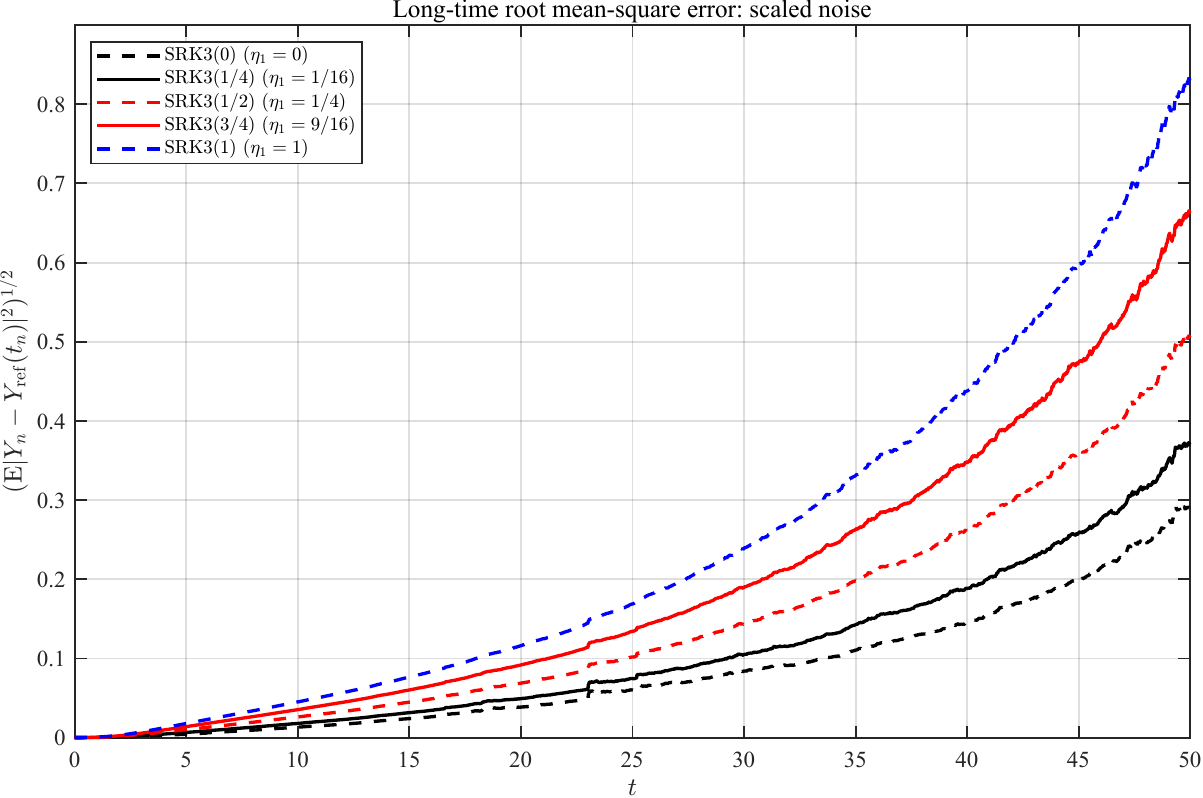}
	\caption{Evolution of the empirical root-mean-square error \eqref{numSRKRMSE} for five members of the family \eqref{parametricSRK}. Upper panel: the two-dimensional equation \eqref{numSRKSDE}, with $Y_0=(0.8,-0.4)^\top$ and $M=5000$. Lower panel: the scalar equation \eqref{numScalarSRKSDE}, with $M=50\,000$. In both tests, $h=0.01$, $h_{\rm ref}=10^{-4}$, and $T=50$; the legends are arranged in increasing order of $\eta_1=\theta^2$.}
	\label{fig:srkLongTimeRMSE}
\end{figure}

We retain $T=50$, $h=0.01$, $h_{\rm ref}=10^{-4}$, $\kappa=3$, and the five values of $\theta$ used above, but increase the sample size to $M=50\,000$ because the solution ensemble is more dispersed at large times. The reference approximation, the coupling of the Wiener increments, and the Newton stopping rule are unchanged, and no nonlinear solve failed. The terminal errors in the lower panel of Figure \ref{fig:srkLongTimeRMSE} are
	$0.2903,~ 0.3699,~ 0.5051,~ 0.6626,~0.8298.$
Thus both equations exhibit the same ordering with respect to $\eta_1$, with a more pronounced separation for \eqref{numScalarSRKSDE}, in accordance with Remark \ref{inf1}.

\section{Conclusion}\label{SecConclusion}
We develop an appurtenant-method framework for deriving asymptotic error distributions of diffusion-implicit schemes by transferring the normalized-error limits from explicit approximations that are sufficiently close to the original methods. The framework is illustrated by the stochastic midpoint method of strong order $1/2$ under general multidimensional noise and applied to fully implicit SRK methods of strong order $1$ with scalar multiplicative noise. For the SRK methods, we derive the limiting error equation, an $\eta_1$-dependent upper bound for the asymptotic root-mean-square error, and a common limiting distribution under the weak-order-two conditions. Numerical experiments support the midpoint limit distribution and illustrate the influence of $\eta_1$ on the long-time root-mean-square error of the SRK family.

\appendix

\vspace{4mm}
\noindent\textbf{Conflict of Interest}
The author reports no conflict of interest.

\vspace{2mm}
\noindent\textbf{Funding}
This work was supported by the National Natural Science Foundation of China (Nos. 12671472 and 12201228).

\vspace{2mm}
\noindent\textbf{Author Contributions}
The author solely conceived, designed, performed the experiments, analyzed the data, and wrote the manuscript.

\vspace{2mm}
\noindent\textbf{Data Availability}
The data supporting the findings of this study are available from the author upon reasonable request.

\bibliography{sn-bibliography}

\end{document}